\documentclass{amsproc}
\usepackage{a4wide,amsfonts,graphicx,
amssymb,stmaryrd,amscd,amsmath,latexsym,amsbsy,xypic}
\xyoption{all}
\usepackage{marginnote}
\numberwithin{equation}{section}
\theoremstyle{plain}

\newtheorem{thm}{Theorem}[section]

\newtheorem{prop}[thm]{Proposition}

\newtheorem{defi}[thm]{Definition}
\newtheorem{lem}[thm]{Lemma}
\newtheorem{cor}[thm]{Corollary}

\newtheorem{eg}[thm]{Example}

\theoremstyle{remark}
\newtheorem{rema}[thm]{Remark}

\begin{document}
\title[Quantum Calogero-Moser spin chains]{Asymptotic boundary KZB operators and quantum Calogero-Moser spin chains}
\author{N. Reshetikhin}
\address{N.R.: Yau Center for Mathematical Sciences, Tsinghua University, Bejing, China \& ITMO University, Kronverskii Ave. 49, Saint Petersburg, 197101, Russia \& KdV Institute for Mathematics, University of Amsterdam, Science Park 105-107, 1098 XG Amsterdam, The Netherlands}
\email{reshetik@math.berkeley.edu}
\thanks{Both authors are supported by the Netherlands Organization for Scientific Research (NWO).  N.R. was partially supported by the grants NSF DMS-1902226 and RFBR No. 18-01-00916. He is also grateful to ETH-ITS Zurich for the hospitality during 2019-2020 academic year. We thank Hadewijch de Clercq, Edward Berengoltz and two referees
for comments and for pointing out a few typos.}
\author{J.V. Stokman}
\address{J.S.: KdV Institute for Mathematics, University of Amsterdam,
Science Park 105-107, 1098 XG Amsterdam, The Netherlands.}
\email{j.v.stokman@uva.nl}
\thanks{}
\begin{abstract}
Asymptotic boundary KZB equations describe
the consistency conditions of degenerations of correlation functions for boundary Wess-Zumino-Witten-Novikov conformal field theory on a cylinder. In the first part of the paper we define asymptotic boundary KZB operators for connected real semisimple Lie groups $G$ with finite center. We prove their main properties algebraically using coordinate versions of Harish-Chandra's radial component map. We show that their commutativity is governed by a system of equations involving coupled versions of classical dynamical Yang-Baxter equations and reflection equations. 

We use the coordinate radial components maps to introduce a new class of quantum superintegrable systems, called quantum Calogero-Moser spin chains. A quantum Calogero-Moser spin chain is a mixture of a quantum spin Calogero-Moser system associated to the restricted root system of $G$ and a one-dimensional spin chain with two-sided reflecting boundaries. 
The asymptotic boundary KZB operators provide explicit expressions for its first-order quantum Hamiltonians. We also explicitly describe the Schr{\"o}dinger operator. 
 \end{abstract}

\maketitle

\section{Introduction}
This paper is part two of a series aimed at connecting harmonic analysis on affine symmetric spaces to boundary Wess-Zumino-Witten-Novikov (WZWN) conformal field theory (see \cite{SR} for part one). 

\subsection{} Let $G$ be a noncompact real connected semisimple Lie group of real rank $r$ with finite center, $K=G^\Theta$ the compact subgroup of fixed points of a Cartan involution $\Theta$ of $G$, and $\mathfrak{h}_{\mathbb{R}}$ a corresponding maximally noncompact Cartan subalgebra of the Lie algebra $\mathfrak{g}_{\mathbb{R}}$ of $G$. We write $\mathfrak{g}$ and $\mathfrak{h}$ for the complexifications of $\mathfrak{g}_{\mathbb{R}}$ and $\mathfrak{h}_{\mathbb{R}}$. Let $\mathcal{M}_{1,n}$ be the moduli space of elliptic curves with $n$ marked points. 

In \cite{FW} Felder and Wieczerkowski introduced the twisted WZWN conformal block bundle over $\mathcal{M}_{1,n}\times \mathfrak{h}$ with flat connection. The explicit realisation of the flat connection leads to  
explicit commuting first-order differential operators, called the Knizhnik-Zamolodchikov-Bernard (KZB) operators \cite{B,FW}. In the limit when the period and marked points go to infinity  the KZB flat connection over $\mathcal{M}_{1,n}\times \mathfrak{h}$ becomes a flat connection over $\mathfrak{h}$. We call the corresponding differential operators the {\it asymptotic KZB operators}. They were studied in \cite{ES}.

In a similar way WZWN conformal field theory on a disc (or on a half plane, or on a strip) with $n$ marked points and Cardy type conformal boundary conditions \cite{C,Z} leads to commuting {\it asymptotic boundary KZB operators} on $\mathfrak{h}$, see \cite{SR}. In this case $G$ is real split, and $K=G^{\Theta}$ is 
the symmetry group at the boundary.
One of the main goals of this paper is to introduce the {\it asymptotic boundary KZB operators for all noncompact real connected semisimple Lie groups $G$ with finite center}.

The resulting asymptotic boundary KZB operators (see Section \ref{SectionabKZB}) are first-order commuting differential operators on $\mathfrak{a}$, where $\mathfrak{a}$ is the complexification of the $(-1)$-eigenspace of $\mathfrak{h}_{\mathbb{R}}$ with respect to the Cartan involution $\theta_{\mathbb{R}}$ of $\mathfrak{g}_{\mathbb{R}}$ associated to $\Theta$.
We construct these asymptotic boundary KZB operators and derive their commutativity 
using generalised Harish-Chandra \cite{HC,CM} radial component maps with respect to the Cartan
decomposition $G=KAK$ of $G$, where $A:=\exp(\mathfrak{a}_{\mathbb{R}})$ (see Section \ref{S4}  for 
the description of the radial component maps). Furthermore, we  show that the asymptotic
boundary KZB operators are particular first-order quantum Hamiltonians for a quantum superintegrable 
system whose algebra of Hamiltonians is isomorphic to a quotient of $Z(U(\mathfrak{g}))^{\otimes (n+1)}$, where $Z(U(\mathfrak{g}))$ denotes the center of the universal enveloping algebra $U(\mathfrak{g})$ (see, e.g., \cite{MPW} for a survey on classical and quantum integrability).
These quantum superintegrable systems are parametrized by a choice of $n$ representations of $G$ and two representations of $K$. We  call them {\it quantum Calogero-Moser spin chains} because  they can be regarded both as a system of interacting "spin" particles and as
a version of a spin chain of Gaudin type with reflecting boundary conditions. 

One of the surprises is the fact that the commutativity of the asymptotic boundary KZB operators is equivalent to a hierarchy of {\it coupled} classical dynamical Yang-Baxter and reflection type equations for the local factors of the operators. These local factors are folded and contracted versions of a natural generalisation \eqref{rmatrix} of Felder's \cite{F} trigonometric dynamical $r$-matrix.

For $G$ real split the coupled integrability equations decouple, reducing them to the mixed classical dynamical Yang-Baxter equations and associated classical dynamical reflection equations from \cite{SR}. This decoupling is caused by the fact that the centraliser $M=Z_K(A)$ of $A$ in $K$ is finite and discrete for real split $G$. In this case the mixed classical dynamical Yang-Baxter equations and the associated classical dynamical reflection equation are a direct consequence of the well-known fact that Felder's trigonometric $r$-matrix satisfies the classical dynamical Yang-Baxter equation
(this will be explained in \cite{St}).

\subsection{} We now describe the results in a bit more detail from the viewpoint of harmonic analysis. Harmonic analysis on symmetric spaces provides a natural representation theoretical context
for a class of quantum superintegrable one-dimensional particle systems called quantum spin Calogero-Moser systems \cite{OP, HO, HS, SR,Re2}. The main observation is as follows.
For two finite-dimensional $K$-representations
$(\sigma_\ell,V_\ell)$, $(\sigma_r,V_r)$ the space $C_{\sigma_\ell,\sigma_r}^\infty(G)$ of spherical functions consists of the smooth functions
$f: G\rightarrow\textup{Hom}(V_r,V_\ell)$ satisfying the equivariance property
\[
f(k_\ell g k_r^{-1})=
\sigma_\ell(k_\ell)f(g)\sigma_r(k_r^{-1})\qquad \forall\, g\in G,\,\,\, \forall\, k_\ell, k_r\in K
\]
(see, e.g., \cite{HC,CM,W}).  A spherical function $f\in C_{\sigma_\ell,\sigma_r}^\infty(G)$ is uniquely
determined by its radial component $f\rvert_A$.
The radial components 
of the biinvariant differential operators acting on $C_{\sigma_\ell,\sigma_r}^\infty(G)$
become vector-valued differential operators on $A$. They form an algebra of commuting differential operators acting on the space $C_{\sigma_\ell,\sigma_r}^\infty(G)\rvert_A$.
This algebra is the algebra of quantum Hamiltonians of the associated quantum spin Calogero-Moser system.
The Schr{\"o}dinger operator of this system corresponds to the quadratic Casimir element 
$\Omega\in Z(U(\mathfrak{g}))$, regarded as a biinvariant differential operator on $C_{\sigma_\ell,\sigma_r}^\infty(G)$. The quantum integrals are the radial components of the action of $U(\mathfrak{g})^{K,\textup{opp}}\otimes_{Z(U(\mathfrak{g}))}U(\mathfrak{g})^K$ on $C_{\sigma_\ell,\sigma_r}^\infty(G)$ by left and right $G$-invariant differential operators, where $U(\mathfrak{g})^K$ is the algebra of $\textup{Ad}(K)$-invariant elements in $U(\mathfrak{g})$ and $U(\mathfrak{g})^{K,\textup{opp}}$ is the associated opposite algebra.

The connection to asymptotic boundary WZWN conformal field theory arises when the left $K$-representation $\sigma_\ell$ is of the form
$(\sigma_{\ell;n},V_\ell\otimes \underline{U})$ with $\underline{U}=U_1\otimes\cdots\otimes U_n$
the tensor product of $n$ finite-dimensional $G$-representations $(\tau_i,U_i)$ and 
$\sigma_{\ell;n}$ the diagonal $K$-action on $V_\ell\otimes\underline{U}$. The analytic Weyl group $W:=N_K(A)/M$ 
naturally acts on $A$, as well as on the space
$(V_\ell\otimes\underline{U}\otimes V_r^*)^M$ of $M$-invariant vectors in 
$V_\ell\otimes\underline{U}\otimes V_r^*$. Restriction to $A$ now defines a linear isomorphism
\begin{equation}\label{ResIso}
\rvert_A: C_{\sigma_{\ell;n},\sigma_r}^\infty(G)\overset{\sim}{\longrightarrow}
C^\infty\bigl(A;(V_\ell\otimes\underline{U}\otimes V_r^*)^M\bigr)^W.
\end{equation}
The space $C^\infty\bigl(A;(V_\ell\otimes\underline{U}\otimes V_r^*)^M\bigr)^W$ is the smooth analogue of the space of asymptotic twisted conformal blocks for boundary WZWN conformal field theory on a 
cylinder \cite{C,Z,FW,ES,SR}, asymptotic in the sense that the insertion points have been sent to infinity within an appropriate asymptotic sector. 
Typical examples of spherical functions in $C_{\sigma_{\ell;n},\sigma_r}^\infty(G)$ are
\begin{equation}\label{Npf}
g\mapsto (\phi_\ell\otimes\textup{id}_{\underline{U}})(\Psi_1\otimes\textup{id}_{U_2\otimes\cdots\otimes U_n})
\cdots (\Psi_{n-1}\otimes\textup{id}_{U_n})\Psi_n\pi_n(g)\phi_r,
\end{equation}
where $(\pi_i,\mathcal{H}_i)$ ($0\leq i\leq n$) 
are smooth $G$-representations, 
$\Psi_i\in\textup{Hom}_G(\mathcal{H}_i, \mathcal{H}_{i-1}\otimes U_i)$ ($1\leq i\leq n$) 
are $G$-intertwiners (asymptotic vertex operators) and 
$\phi_\ell\in\textup{Hom}_K(\mathcal{H}_0,V_\ell)$, 
$\phi_r\in\textup{Hom}_K(V_r,\mathcal{H}_n)$ are $K$-intertwiners (asymptotic 
boundary states). Functions of the form \eqref{Npf} appear naturally as limits of $n$-point functions in twisted conformal blocks for boundary WZWN conformal field theory on a 
cylinder, cf. \cite{ES,FW}. The twisting corresponds to the insertion of $\pi_n(g)$ in \eqref{Npf}. 

\subsection{} Replacing the $\Psi_i$ in \eqref{Npf} by secondary asymptotic field operators $(\pi_{i-1}(g_{i-1})\otimes\textup{id}_{U_i})\Psi_i$ is providing smooth $V_\ell\otimes\underline{U}\otimes V_r^*$-valued functions on $G^{\times (n+1)}$, equivariant with respect to the $K\times G^{\times n}\times K$-action \eqref{gaugeaction} on $G^{\times (n+1)}$. The corresponding vector space $C_{\sigma_\ell,\underline{\tau},\sigma_r}(G^{\times (n+1)})$ of $K\times G^{\times n}\times K$-equivariant functions is naturally isomorphic to $C_{\sigma_{\ell;n},\sigma_r}^\infty(G)$,
\begin{equation}\label{isointro}
C_{\sigma_\ell,\underline{\tau},\sigma_r}^\infty(G^{\times (n+1)})\overset{\sim}{\longrightarrow}
C_{\sigma_{\ell;n},\sigma_r}^\infty(G),\qquad f\mapsto f^\flat
\end{equation}
with $f^\flat(g):=f(1,\ldots,1,g)$. The upshot of this observation is that $C_{\sigma_\ell,\underline{\tau},\sigma_r}^\infty(G^{\times (n+1)})$ admits a natural 
$U(\mathfrak{g})^{K,\textup{opp}}\otimes Z(U(\mathfrak{g}))^{\otimes (n-1)}\otimes U(\mathfrak{g})^K$-action in terms of coordinate-wise
 invariant differential operators. 
For the commutative subalgebra $Z(U(\mathfrak{g}))^{\otimes (n+1)}$ the action is given in terms of  coordinate-wise
biinvariant differential operators. 
Pushing this action through the two isomorphisms \eqref{isointro} and 
\eqref{ResIso} provides algebras $A^{\sigma_\ell,\underline{\tau},\sigma_r}\subseteq B^{\sigma_\ell,\underline{\tau},\sigma_r}$ 
of $W$-invariant differential operators on $A$ with coefficients in $\textup{End}((V_\ell\otimes\underline{U}\otimes V_r^*)^M)$, with $A^{\sigma_\ell,\underline{\tau},\sigma_r}$
commutative. They serve, after an appropriate gauge, as the algebras of quantum Hamiltonians and quantum integrals for a quantum superintegrable system, which we call the quantum Calogero-Moser spin chain.

Gauged radial components of the action of the Casimir element $\Omega\in Z(U(\mathfrak{g}))$ as biinvariant differential operator on the $j$th coordinate of $C_{\sigma_\ell,\underline{\tau},\sigma_r}^\infty(G^{\times (n+1)})$ ($0\leq j\leq n$) are providing commuting second-order quantum Hamiltonians $\widetilde{D}_{\Omega,j;M}^{\sigma_\ell,\underline{\tau},\sigma_r}$ ($0\leq j\leq n$). 
We will derive explicit expressions for $\widetilde{D}_{\Omega,j;M}^{\sigma_\ell,\underline{\tau},\sigma_r}$ by computing $\widetilde{D}_{\Omega,n;M}^{\sigma_\ell,\underline{\tau},\sigma_r}$, which serves as the Schr{\"o}dinger operator of the quantum Calogero-Moser spin chain, 
 and the differences 
\[
\mathcal{\widetilde{D}}_{i;M}^{\sigma_\ell,\underline{\tau},\sigma_r}:=
\widetilde{D}_{\Omega,i;M}^{\sigma_\ell,\underline{\tau},\sigma_r}-
\widetilde{D}_{\Omega,i-1;M}^{\sigma_\ell,\underline{\tau},\sigma_r}
\qquad (1\leq i\leq n).
\]
The vector-valued potential $V^{\sigma_\ell,\underline{\tau},\sigma_r}$ (see \eqref{potential})
of the Schr{\"o}dinger operator $\widetilde{D}_{\Omega,n;M}^{\sigma_\ell,\underline{\tau},\sigma_r}$, which we will compute explicitly using standard techniques \cite{HC,CM,HO} from harmonic analysis on symmetric spaces, has the typical pairwise $\textup{sinh}^{-2}$-interaction terms associated to the roots of the restricted root system of $G$.
 The first-order commuting differential operators $\mathcal{\widetilde{D}}_{i;M}^{\sigma_\ell,\underline{\tau},\sigma_r}$ are the asymptotic boundary KZB operators.
They are of the form
\[
\widetilde{\mathcal{D}}_{i;M}^{\sigma_\ell,\underline{\tau},\sigma_r}=\sum_{j=1}^r\tau_i(x_j)\partial_{x_j}-\widehat{\kappa}_i-\sum_{k=1}^{i-1}r_{ki}^+-
\sum_{k=i+1}^Nr_{ik}^-
\]  
with $\{x_i\}_{i=1}^r$ an orthonormal basis of 
$\mathfrak{a}_{\mathbb{R}}$. The local factors $\kappa_i, r_{ki}^+, r_{ik}^-$ are $\textup{End}((V_\ell\otimes\underline{U}\otimes V_r^*)^M)$-valued functions on $\mathfrak{a}_{\mathbb{R}}$, acting
nontrivially on the tensor components $V_\ell\otimes U_i\otimes V_r^*$, $U_k\otimes U_i$ and $U_i\otimes U_k$ within
$(V_\ell\otimes\underline{U}\otimes V_r^*)^M$, respectively.
We will show that the operators $r^{\pm}_{ik}(a)$ are given by the action on $U_i\otimes U_k$ of $\theta$-(anti)symmetrised versions of the non-split analogue of Felder's \cite{F} classical dynamical $r$-matrix $r(a)$, and the core component $\widehat{\kappa}^{\textup{core}}_i(a)$ of $\widehat{\kappa}_i(a)$ is the action on $U_i$ of a $\theta$-twisted contraction of $r(a)$  (see Proposition \ref{folding} and \eqref{kappacorealt}). 

As mentioned earlier, the operators $\widehat{\kappa}_i,r_{ki}^+$ and $r_{ik}^-$ are solutions of a hierarchy of integrability equations 
(see Theorem \ref{inteq}). It contains a classical dynamical reflection type equation for $\widehat{\kappa}$ relative to $(r^+,r^-)$ (see \eqref{cdRe}) and three coupled classical dynamical Yang-Baxter-reflection type equations for the triple $(r^+,r^-,\widehat{\kappa})$ (see \eqref{cdYBRe}).

In \cite{SR} formal $n$-point spherical functions were introduced when $G$ is real split. 
They provide common eigenfunctions for all the quantum Hamiltonians  after a suitable gauge. In \cite{SR} this approach was used to arrive at the explicit form of the asymptotic boundary KZB operators, and to prove their commutativity.
The theory of formal $n$-point spherical functions for non-split $G$ is yet to be developed, but we
do shortly discuss global $n$-point spherical functions.

The results and techniques in this paper are closely related to the Etingof-Schiffmann-Varchenko theory on generalised trace functions, see, e.g., \cite{ES,EV00,ESV,EL} (our current trigonometric, non-affine level of the theory was discussed for generalised trace functions in \cite{ES}). {}From the harmonic analysis point of view the Etingof-Schiffmann-Varchenko theory is related to the symmetric pair $(G\times G,\textup{diag}(G))$, with $\textup{diag}(G)$ the diagonal embedding of $G$ into $G\times G$. The extension of the results in the current paper to the level of (quantum) affine symmetric pairs is currently under investigation.\\

\noindent
The content of the paper is as follows. 

In Section 2 we introduce the space of global $n$-point spherical functions and study the coordinate-wise action of (bi)invariant differential operators on $C_{\sigma_{\ell},\underline{\tau},\sigma_r}^\infty(G^{\times (n+1)})\simeq C_{\sigma_{\ell;n},\sigma_r}^\infty(G)$.
This part of the paper, which does not yet require radial component maps,
will be developed for the more general context involving an arbitrary real connected Lie group $G$ and two closed Lie subgroups $K_\ell$ and $K_r$. 

Section 3 is a short section in which we recall some basic facts about the structure theory of real semisimple Lie groups. This section is mainly meant to fix notations.

In Section 4 we 
introduce the coordinate radial component maps and prove their main algebraic properties. 

In Section 5 we use the coordinate radial component maps to define the algebras of quantum Hamiltonians and quantum integrals for the quantum Calogero-Moser spin chain. We furthermore derive the explicit expression for the Schr{\"o}dinger operator of the quantum Calogero-Moser spin chain. 

In Section 6 we derive the explicit expressions for the asymptotic boundary KZB operators and we
show that the local factors of the asymptotic boundary KZB operators satisfy coupled classical dynamical Yang-Baxter-reflection equations. 

Finally, in Section 7 we explicitly compute the main structural ingredients in case $G=\textup{SU}(p,r)$ with $1\leq r\leq p$ (in this case the underlying restricted root system is of type $\textup{BC}_r$).\\

\noindent
{\bf Conventions:} We write $\otimes_{\mathcal{A}}$ for the tensor product over a complex associative algebra $\mathcal{A}$. For $\mathcal{A}=\mathbb{C}$ we simply denote it by $\otimes$. A similar convention is used for hom-spaces; $\textup{Hom}_{\mathcal{A}}(E,F)$ denotes the space of $\mathcal{A}$-linear maps, and $\textup{Hom}(E,F)$ the space of complex linear maps.

\section{$n$-Point spherical functions}\label{S3}
Let $G$ be a real connected Lie group, and $K_\ell, K_r\subseteq G$ two closed Lie subgroups.
Write for $n\in\mathbb{Z}_{\geq 1}$,
\[
\mathcal{G}_n:=K_\ell\times G^{\times n}\times K_r.
\]
Elements in $G^{\times n}$ will be denoted by $\underline{h}=(h_1,\ldots,h_n)$, and elements
in $G^{\times (n+1)}$ by $\mathbf{g}=(g_0,\ldots,g_n)$.
Consider the left $\mathcal{G}_n$-action 
on $G^{\times (n+1)}$ defined by
\begin{equation}\label{gaugeaction}
(k_\ell,\underline{h},k_r)\cdot \mathbf{g}:=(k_\ell g_0h_1^{-1},h_1g_1h_2^{-1},
\ldots, h_ng_nk_r^{-1}).
\end{equation}
For finite-dimensional smooth $G$-representations $(\tau_i, U_i)$ ($1\leq i\leq n$) write
$\underline{\tau}=\tau_1\otimes\cdots\otimes\tau_n$ for the tensor product representation of
$G^{\times n}$ on $\underline{U}:=U_1\otimes\cdots\otimes U_n$. 
Furthermore, fix finite-dimensional smooth representations $(\sigma_\ell, V_\ell)$, $(\sigma_r, V_r)$ of $K_\ell$ and $K_r$ respectively, and write $(\sigma_r^*,V_r^*)$ for
the $K_r$-representation dual to $(\sigma_r,V_r)$.
\begin{defi}\label{sN}
Let $C_{\sigma_\ell,\underline{\tau},\sigma_r}^\infty(G^{\times (n+1)})$ be the space of
smooth functions 
\[
f: G^{\times (n+1)}\rightarrow V_\ell\otimes\underline{U}\otimes V_r^*\]
satisfying 
\[
f((k_\ell,\underline{h},k_r)\cdot\mathbf{g})=
(\sigma_\ell(k_\ell)\otimes\underline{\tau}(\underline{h})\otimes\sigma_r^*(k_r))
f(\mathbf{g})\qquad \forall\, (k_\ell,\underline{h},k_r)\in \mathcal{G}_n,\,\,\, \forall\, \mathbf{g}\in
G^{\times (n+1)}.
\]
\end{defi}
\begin{rema}
For $n=0$ the space $C_{\sigma_\ell,\sigma_r}^{\infty}(G)$ consists of the
smooth functions $f: G\rightarrow V_\ell\otimes V_r^*$ satisfying
\[
f(k_\ell gk_r^{-1})=(\sigma_\ell(k_\ell)\otimes\sigma_r^*(k_r))f(g)
\qquad \forall\, k_\ell\in K_\ell,\,\, \forall\, k_r\in K_r,\,\, \forall\, g\in G.
\]
When $K_\ell=K_r$, the space $C_{\sigma_\ell,\sigma_r}^{\infty}(G)$ is called the space of $\sigma_\ell\otimes\sigma_r^*$-spherical functions on $G$ with respect to the Lie subgroup $K_\ell$
of $G$.
\end{rema}
\begin{rema}\label{conv}
We will use standard tensor-leg notations.
For example, for $v_\ell\in V_\ell$, $\underline{u}\in\underline{U}$, $\phi_r\in V_r^*$ and $g_1,g_2\in G$ the tensor
\[
(\textup{id}_{V_\ell}\otimes\tau_1(g_1)\otimes\tau_2(g_2)\otimes\textup{id}_{U_3}\otimes
\cdots\otimes\textup{id}_{U_n}\otimes\textup{id}_{V_r^*})(v_\ell\otimes\underline{u}\otimes\phi_r)
\]
in $V_\ell\otimes\underline{U}\otimes V_r^*$
will be denoted by $(\tau_1(g_1)\otimes \tau_2(g_2))(v_\ell\otimes\underline{u}\otimes\phi_r)$, or
$\tau_1(g_1)\tau_2(g_2)(v_\ell\otimes\underline{u}\otimes\phi_r)$.
We will write omitted components in vectors with a cap. For instance, 
$(g_0,\ldots,\hat{g}_j,\ldots,g_n)$ stands for the $n$-vector 
$(g_0,\ldots,g_{j-1},g_{j+1},\ldots,g_n)$ in $G^{\times n}$.
\end{rema}

We write $(\sigma_{\ell;n}, V_\ell\otimes\underline{U})$ for the finite-dimensional 
$K_\ell$-representation defined by
\[
\sigma_{\ell;n}(k_\ell):=\sigma_\ell(k_\ell)\otimes\tau_1(k_\ell)\otimes\cdots\otimes\tau_n(k_\ell)
\qquad \forall\, k_\ell\in K_\ell.
\]

For a smooth function $f: G^{\times (n+1)}\rightarrow V_\ell\otimes\underline{U}\otimes V_r^*$
we define $f^\flat: G\rightarrow V_\ell\otimes\underline{U}\otimes V_r^*$ by
\begin{equation}\label{tildef}
f^\flat(g):=f(1,\ldots,1,g)\qquad (g\in G).
\end{equation}
\begin{lem}\label{Nto0}
The linear map $f\mapsto f^\flat$ restricts to a linear isomorphism 
\[
C_{\sigma_\ell,\underline{\tau},\sigma_r}^\infty(G^{\times (n+1)})
\overset{\sim}{\longrightarrow} C_{\sigma_{\ell;n},\sigma_r}^\infty(G)
\]
The preimage of $f^\flat\in C_{\sigma_{\ell;n},\sigma_r}^\infty(G)$ is the function
$f\in C_{\sigma_\ell,\underline{\tau},\sigma_r}^\infty(G^{\times (n+1)})$ defined by
\begin{equation}\label{inverse}
f(\mathbf{g}):=(\tau_1(g_0^{-1})\otimes\tau_2(g_1^{-1}g_0^{-1})\otimes\cdots
\otimes\tau_n(g_{n-1}^{-1}\cdots g_0^{-1}))f^\flat(g_0g_1\cdots g_n)
\end{equation} 
for $\mathbf{g}\in G^{\times (n+1)}$.
\end{lem}
\begin{proof}
Note that $f^\flat\in C_{\sigma_{\ell;n},\sigma_r}^\infty(G)$ for
$f\in C_{\sigma_\ell,\underline{\tau},\sigma_r}^\infty(G^{\times (n+1)})$ since
\[
(k_\ell,(k_\ell,\ldots,k_\ell),k_r)\cdot (1,\ldots,1,g)=(1,\ldots,1,k_\ell g k_r^{-1})
\]
in $G^{\times (n+1)}$ for $k_\ell\in K_\ell$, $k_r\in K_r$ and $g\in G$.
Conversely, for
$f^\flat\in C_{\sigma_{\ell;n},\sigma_r}^\infty(G)$, it follows from the formula
\[
(1,g_0,g_0g_1,\ldots,g_0g_1\cdots g_{n-1},1)\cdot \mathbf{g}=
(1,\ldots,1,g_0g_1\cdots g_n)\qquad (\mathbf{g}\in G^{\times (n+1)})
\]
that $f$, defined by \eqref{inverse}, lies in $C_{\sigma_\ell,\underline{\tau},\sigma_r}^\infty(G^{\times (n+1)})$. The lemma now follows easily.
\end{proof}
For a complex Lie algebra $L$ we write $U(L)$ for its universal enveloping algebra. It is a Hopf algebra, with counit $\epsilon: U(L)\rightarrow\mathbb{C}$ the unital algebra homomorphism
satisfying $\epsilon(y)=0$ ($y\in L$), comultiplication
$\Delta: U(L)\rightarrow U(L)\otimes U(L)$ the unital algebra homomorphism satisfying $\Delta(y)=y\otimes 1+1\otimes y$ ($y\in L$), and with antipode $S: U(L)\rightarrow U(L)$ the unital algebra anti-isomorphism satisfying $S(y)=-y$ ($y\in L$). For $k\in\mathbb{Z}_{>0}$ let 
$\Delta^{(k)}: U(L)\rightarrow U(L)^{\otimes (k+1)}$ be the $k$th iterated comultiplication, defined
inductively by $\Delta^{(1)}=\Delta$ and $\Delta^{(k)}=(\Delta\otimes\textup{id}_{U(L)}^{\otimes (k-1)})\Delta^{(k-1)}$ for $k>1$. We will use Sweedler's notation for the expression of $\Delta^{(k)}(u)$ ($u\in U(L)$) as sum of pure tensors,
\[
\Delta^{(k)}(u)=\sum_{(u)}u_{(1)}\otimes\cdots\otimes u_{(k+1)}.
\]
For a finite-dimensional $U(L)$-module $(\sigma,V)$ we write $(\sigma^*,V^*)$ for the dual $U(L)$-module, defined
by $(\sigma^*(u)f)(v)=-f(\sigma(u)v)$ for $u\in U(L)$, $v\in V$ and $f\in V^*$. We write $V^{L}$ for the space of $L$-invariant elements in $V$. 

Write $\mathfrak{g}_{\mathbb{R}}$, $\mathfrak{k}_{\ell,\mathbb{R}}$ and $\mathfrak{k}_{r,\mathbb{R}}$ for the Lie algebras
of $G$, $K_\ell$ and $K_r$ respectively. Their complexifications are denoted by 
$\mathfrak{g}$, $\mathfrak{k}_\ell$ and $\mathfrak{k}_r$. Differentiating the
representations $(\sigma_\ell,V_\ell)$, $(\sigma_r,V_r)$ and $(\tau_i,U_i)$ 
turns $V_\ell$ into a left $U(\mathfrak{k}_\ell)$-module, $V_r$ into a left 
$U(\mathfrak{k}_r)$-module and the $U_i$ into left
$U(\mathfrak{g})$-modules. The corresponding representation maps will again be denoted by
$\sigma_\ell, \sigma_r$ and $\tau_i$. 

Let $V$ be a finite-dimensional complex vector space. We denote by $C^\infty(G^{\times (n+1)};V)$
the space of smooth $V$-valued functions on $G^{\times (n+1)}$. We write $C^\infty(G^{\times (n+1)})$ when $V=\mathbb{C}$.

We have a left $U(\mathfrak{g})^{\otimes (n+1)}$-action on $C^\infty(G^{\times (n+1)};V)$,
defined by
\[
u_0\otimes\cdots\otimes u_n\mapsto u_0[0]u_1[1]\cdots u_n[n]\qquad (u_j\in U(\mathfrak{g}))
\]
with $u\mapsto u[j]$ ($u\in U(\mathfrak{g})$) the action of $U(\mathfrak{g})$ on the $j^{th}$ coordinate of
$f\in C^\infty(G^{\times (n+1)};V)$ by left $G$-invariant differential operators. Concretely, the action $u\mapsto u[j]$ of $U(\mathfrak{g})$ on
$C^\infty(G^{\times (n+1)};V)$ is determined by
\[
(y[j]f)(\mathbf{g}):=\frac{d}{dt}\biggr\rvert_{t=0}f(g_0,\ldots,g_{j-1},g_j\exp(ty),g_{j+1},\ldots,g_n)
\qquad (y\in\mathfrak{g}_{\mathbb{R}}),
\]
where $\exp: \mathfrak{g}_{\mathbb{R}}\rightarrow G$ is the exponential map. 
\begin{rema}\label{useful}
We write $u\mapsto u\langle j\rangle$ for the coordinate-wise action of $U(\mathfrak{g})$
on $C^\infty(G^{\times (n+1)};V)$ by right $G$-invariant differential operators, so that
\[
(y\langle j\rangle f)(\mathbf{g}):=\frac{d}{dt}\biggr\rvert_{t=0}f(g_0,\ldots,g_{j-1},\exp(-ty)g_j,g_{j+1},\ldots,g_n)
\qquad (y\in\mathfrak{g}_{\mathbb{R}}).
\]
Note that 
\begin{equation*}
\bigl(u\langle j\rangle f\bigr)(\mathbf{g})=\bigl(\textup{Ad}_{g_j^{-1}}(S(u))[j]f\bigr)(\mathbf{g})
\end{equation*}
for $u\in U(\mathfrak{g})$. In particular, $u\langle j\rangle f=S(u)[j]f$ for $u\in Z(U(\mathfrak{g}))$,
with $Z(U(\mathfrak{g}))$ the center of $U(\mathfrak{g})$.
\end{rema}

Let $C_g$ be the conjugation action of $g\in G$ on $G$, defined by
$C_g(h):=ghg^{-1}$ for $g,h\in G$. Differentiating $C_g$ defines the adjoint action 
$\textup{Ad}_g\in\textup{Aut}(\mathfrak{g}_{\mathbb{R}})$ of $g\in G$ on $\mathfrak{g}_{\mathbb{R}}$. It extends to an action of $G$ on $U(\mathfrak{g})$ by unital algebra automorphisms, which will also be denoted by $g\mapsto \textup{Ad}_g$.
\begin{prop}\label{transfogen}
For $f\in C_{\sigma_\ell,\underline{\tau},\sigma_r}^\infty(G^{\times (n+1)})$
and $0\leq j<n$ we have
\begin{equation*}
\bigl(u[j]f\bigr)(\mathbf{g})=
\sum_{(u)}A_{j,(u)}(\mathbf{g})\bigl(\textup{Ad}_{g_n^{-1}\cdots g_{j+1}^{-1}}(u_{(n-j+1)})[n]f\bigr)(
g_0,\ldots,\hat{g}_{j},\ldots,g_n,C_{g_n^{-1}\cdots g_{j+1}^{-1}}(g_j)\bigr)
\end{equation*}
for $\mathbf{g}\in G^{\times (n+1)}$ and $u\in U(\mathfrak{g})$, with
$A_{j,(u)}(\mathbf{g})\in\textup{End}(V_\ell\otimes\underline{U}\otimes V_r^*)$ given by
\begin{equation*}
\begin{split}
A_{j,(u)}(\mathbf{g}):=\tau_{j+1}(S(u_{(1)}))\tau_{j+1}(g_j^{-1}&g_{j+1})
\otimes\tau_{j+2}(g_{j+1}^{-1})\tau_{j+2}(S(u_{(2)}))\tau_{j+2}(g_j^{-1}g_{j+1}g_{j+2})\otimes\cdots\\
&\,\,\,\,\,\,\cdots\otimes\tau_n(g_{n-1}^{-1}\cdots g_{j+1}^{-1})\tau_n(S(u_{(n-j)}))
\tau_n(g_j^{-1}g_{j+1}\cdots g_n).
\end{split}
\end{equation*}
\end{prop}
\begin{proof}
For the duration of the proof we use the shorthand notation
\[
F^{(j)}(\mathbf{g}):=F(g_0,\ldots,\hat{g}_j,\ldots,g_n,C_{g_n^{-1}\cdots g_{j+1}^{-1}}(g_j))
\]
for $F\in C^\infty(G^{\times (n+1)};
V_\ell\otimes\underline{U}\otimes V_r^*)$ and $0\leq j<n$. One then easily checks that for $y\in\mathfrak{g}_{\mathbb{R}}$,
\begin{equation}\label{help}
\bigl(y[j]F^{(j)}\bigr)(\mathbf{g})=
\bigl(\textup{Ad}_{g_n^{-1}\cdots g_{j+1}^{-1}}(y)[n]F\bigr)^{(j)}(\mathbf{g}).
\end{equation}

Consider now the standard filtration $U(\mathfrak{g})=
\bigcup_{k\in\mathbb{Z}_{\geq 0}}U_k(\mathfrak{g})$ of $U(\mathfrak{g})$.
We prove the proposition for $u\in U_k(\mathfrak{g})$ by induction in $k$.
For $k=0$ the result follows from the formula
\[
f(\mathbf{g})=\bigl(\tau_{j+1}(g_j^{-1}g_{j+1})\otimes\cdots\otimes\tau_n(g_{n-1}^{-1}\cdots g_j^{-1}g_{j+1}\cdots g_n)\bigr)f^{(j)}(\mathbf{g}),
\] 
which holds true since $f\in C_{\sigma_\ell,\underline{\tau},\sigma_r}^\infty(G^{\times (n+1)})$.
For the induction step, suppose the proposition is correct for $u\in U_k(\mathfrak{g})$, so
\[
\bigl(u[j]f\bigr)(\mathbf{g})=
\sum_{(u)}A_{j,(u)}(\mathbf{g})\bigl(\textup{Ad}_{g_n^{-1}\cdots g_{j+1}^{-1}}(u_{(n-j+1)})[n]f\bigr)^{(j)}(\mathbf{g}).
\]
Then for $y\in\mathfrak{g}_{\mathbb{R}}$ we have by \eqref{help},
\begin{equation}\label{help2}
\begin{split}
\bigl((yu)[j]f\bigr)(\mathbf{g})&=
\sum_{(u)}\bigl(y[j]A_{j,(u)}\bigr)(\mathbf{g})\bigl(\textup{Ad}_{g_n^{-1}\cdots g_{j+1}^{-1}}(u_{(n-j+1)})[n]f\bigr)^{(j)}(\mathbf{g})\\
&+\sum_{(u)}A_{j,(u)}(\mathbf{g})\bigl(\textup{Ad}_{g_n^{-1}\cdots g_{j+1}^{-1}}(yu_{(n-j+1)})[n]f\bigr)^{(j)}(\mathbf{g}).
\end{split}
\end{equation}
By the explicit expression for $A_{j,(u)}(\mathbf{g})$ and the product rule we have
\[
\bigl(y[j]A_{j,(u)}\bigr)(\mathbf{g})=
\sum_{s=j+1}^n\Bigl(
\cdots\otimes
\tau_s(g_{s-1}^{-1}\dots g_{j+1}^{-1})\tau_s(S(yu_{(s-j)}))\tau_s(g_j^{-1}g_{j+1}\cdots g_s)\otimes
\cdots\Bigr)
\]
where we specified only the tensor component that differs from the corresponding tensor component of $A_{j,(u)}(\mathbf{g})$. Substituting into \eqref{help2} proves the proposition for
$yu\in U_{k+1}(\mathfrak{g})$. This completes the proof.
\end{proof}
\begin{cor}\label{corflat}
For $f\in C_{\sigma_\ell,\underline{\tau},\sigma_r}^\infty(G^{\times (n+1)})$
and $0\leq j<n$ we have
\[
(u[j]f)^{\flat}(g)=\sum_{(u)}\tau_{j+1}(S(u_{(1)}))\cdots\tau_n(S(u_{(n-j)}))
\bigl(S(u_{(n-j+1)})\langle n\rangle f\bigr)^{\flat}(g)
\]
for $g\in G$ and $u\in U(\mathfrak{g})$.
\end{cor}
\begin{proof}
By Proposition \ref{transfogen},
\[
(u[j]f)^\flat(g)=\sum_{(u)}\tau_{j+1}(S(u_{(1)}))\cdots\tau_n(S(u_{(n-j)}))
\tau_n(g)\bigl(\textup{Ad}_{g^{-1}}(u_{(n-j+1)})[n]f\bigr)(1,\ldots,1,g,1).
\]
The result now follows from Remark \ref{useful} and the observation that
\[
\tau_n(g)f(g_0,\ldots,g_{n-2},g,g_n)=f(g_0,\ldots,g_{n-2},1,gg_n)
\]
since  $f\in C_{\sigma_\ell,\underline{\tau},\sigma_r}^\infty(G^{\times (n+1)})$.
\end{proof}
\begin{rema}\label{remflat}
A similar analysis can be done for the coordinate-wise action of $U(\mathfrak{g})$ 
by right $G$-invariant differential operators. It leads to the following analogue of Corollary
\ref{corflat}. Suppose that $f\in C_{\sigma_\ell,\underline{\tau},\sigma_r}^\infty(G^{\times (n+1)})$
and $0\leq j<n$, then
\[
(u\langle j\rangle f)^\flat(g)=\sum_{(u)}\tau_{j+1}(u_{(1)})\cdots\tau_n(u_{(n-j)})
\bigl(u_{(n-j+1)}\langle n\rangle f\bigr)^{\flat}(g)
\]
for $g\in G$ and $u\in U(\mathfrak{g})$.
\end{rema}

For a closed Lie subgroup $K\subseteq G$ write $U(\mathfrak{g})^{K}$
for the subalgebra of $\textup{Ad}(K)$-invariant elements in $U(\mathfrak{g})$,
and $U(\mathfrak{g})^{K,\textup{opp}}$ for the opposite algebra. Note that $S: U(\mathfrak{g})^K\rightarrow U(\mathfrak{g})^{K,\textup{opp}}$ is an isomorphism of algebras.
\begin{lem}\label{qiaction}
The space $C_{\sigma_\ell,\underline{\tau},\sigma_r}^\infty(G^{\times (n+1)})$ is a 
$U(\mathfrak{g})^{K_\ell,\textup{opp}}\otimes Z(U(\mathfrak{g}))^{\otimes (n-1)}\otimes
U(\mathfrak{g})^{K_r}$-module, with the action defined by 
\begin{equation}\label{actionoverall}
\bigl(v_\ell\otimes (u_1\otimes\cdots\otimes u_{n-1})\otimes v_r\bigr)\ast f:= 
S(v_\ell)\langle 0\rangle u_1[1]\cdots u_{n-1}[n-1]v_r[n]f
\end{equation}
for $v_\ell\in U(\mathfrak{g})^{K_\ell}$, $u_i\in Z(U(\mathfrak{g}))$ \textup{(}$1\leq i<n$\textup{)}, $v_r\in U(\mathfrak{g})^{K_r}$ and $f\in C_{\sigma_\ell,\underline{\tau},\sigma_r}^\infty(G^{\times (n+1)})$. Furthermore, for $u_0\otimes\cdots\otimes u_n\in Z(U(\mathfrak{g}))^{\otimes (n+1)}$ we have
the action
\[
(u_0\otimes\cdots\otimes u_n)\ast f=u_0[0]u_1[1]\cdots u_n[n]f\qquad 
(f\in C_{\sigma_\ell,\underline{\tau},\sigma_r}^\infty(G^{\times (n+1)}))
\]
by coordinate-wise biinvariant differential operators.
\end{lem}
\begin{proof}
The straightforward proof is left to the reader.
\end{proof}
\begin{rema}\label{qiactionrem}
For $n=0$ the appropriate analogue of Lemma \ref{qiaction} is the statement that
$U(\mathfrak{g})^{K_\ell,\textup{opp}}\otimes_{Z(U(\mathfrak{g}))}U(\mathfrak{g})^{K_r}$ acts on $C_{\sigma_\ell,\sigma_r}^\infty(G)$ by 
\[
(v_\ell\otimes_{Z(U(\mathfrak{g}))}v_r)\ast f=S(v_\ell)\langle 1\rangle v_r[1]f\qquad
(v_\ell\in U(\mathfrak{g})^{K_\ell}, v_r\in U(\mathfrak{g})^{K_r})
\]
for $f\in C_{\sigma_\ell,\sigma_r}^{\infty}(G)$, where $u\langle 1\rangle f$ and $u[1]f$ now denote the action of $u\in U(\mathfrak{g})$ on $f\in C^\infty(G)$ as right and left $G$-invariant differential operator,
respectively. This is a well-defined action since $S(u)\langle 1\rangle f=u[1]f$ for $u\in Z(U(\mathfrak{g}))$, cf. Remark \ref{useful}
(it is for this reason that we have taken the opposite algebra $U(\mathfrak{g})^{K_\ell,\textup{opp}}$
in Lemma \ref{qiaction}).
\end{rema}
We can now extend the definition of $n$-point spherical functions from \cite{SR} to the current context. 
\begin{defi}\label{defnpoint}
We call $f\in C_{\sigma_\ell,\underline{\tau},\sigma_r}^\infty(G^{\times (n+1)})$
an $n$-point spherical function if there exist algebra homomorphisms $\chi_j: Z(U(\mathfrak{g}))\rightarrow\mathbb{C}$ \textup{(}$0\leq j\leq n$\textup{)} such that
\begin{equation}\label{DOj}
u[j]f=\chi_j(u)f\qquad \forall\, u\in Z(U(\mathfrak{g})),\,\, \forall\, j\in\{0,\ldots,n\}.
\end{equation}
We write
\[
C_{\sigma_\ell,\underline{\tau},
\sigma_r}^\infty(G^{\times (n+1)};\pmb{\chi})\subseteq
C_{\sigma_\ell,\underline{\tau},
\sigma_r}^\infty(G^{\times (n+1)})
\]
for the $U(\mathfrak{g})^{K_\ell}\otimes Z(U(\mathfrak{g}))^{\otimes (n-1)}\otimes U(\mathfrak{g})^{K_r}$-module
of $n$-point spherical functions
satisfying \eqref{DOj} with respect to the $(n+1)$-tuple $\pmb{\chi}:=(\chi_0,\ldots,\chi_n)$
of characters of $Z(U(\mathfrak{g}))$.
\end{defi}

Examples of $n$-point spherical functions $f\in C_{\sigma_\ell,\underline{\tau},\sigma_r}^\infty(G^{\times (n+1)};\pmb{\chi})$ can be constructed using 
$G$-intertwiners $\Psi_i: \mathcal{H}_i\rightarrow \mathcal{H}_{i-1}\otimes U_i$ ($1\leq i\leq n$), with $(\pi_j,\mathcal{H}_j)$ ($0\leq j\leq n$) quasi-simple smooth $G$-representations with central character $\chi_j$,
together with a $K_\ell$-intertwiner $\phi_\ell: \mathcal{H}_0\rightarrow V_\ell$ and a
$K_r$-intertwiner $\phi_r: V_r\rightarrow \mathcal{H}_n$, via the formula
\[
f(\mathbf{g}):=(\phi_\ell\pi_0(g_0)\otimes\textup{id}_{\underline{U}})
(\Psi_1\pi_1(g_1)\otimes\textup{id}_{U_2\otimes\cdots\otimes U_n})
\cdots (\Psi_{n-1}\pi_{n-1}(g_{n-1})\otimes\textup{id}_{U_n})
\Psi_n\pi_n(g_n)\phi_r.
\]
 See \cite{SR} for important explicit classes of such examples.
  
 For special choices of $K_\ell$ and $K_r$, the action of the commuting differential operators
 $x[j]$ ($x\in Z(U(\mathfrak{g}))$, $0\leq j\leq N$) on
 $C_{\sigma_\ell,\underline{\tau},\sigma_r}^\infty(G^{\times (n+1)})$
 gives rise to commuting vector-valued differential operators 
  on a suitable abelian Lie subgroup
 $A\subseteq G$ by pulling the invariant differential operators through the restriction 
 map $f\mapsto f^\flat|_A$ 
 ($f\in C_{\sigma_\ell,\underline{\tau},\sigma_r}^\infty(G^{\times (n+1)})$). Typically, one needs
a natural parametrisation of the double $(K_\ell,K_r)$-cosets in $G$ in terms of orbits in $A$
with respect to a discrete group action on $A$.
In the remainder of this paper we give a detailed analysis when $K:=K_\ell=K_r$ is a maximal compact Lie subgroup of a real connected semisimple Lie group $G$ with finite center. 
 
\section{Structure theory of real semisimple Lie groups}\label{S2}
We introduce standard facts and notations regarding the structure theory of real semisimple Lie groups. For more information see, e.g., \cite[Chpt. VI]{Kn}, \cite[Chpt. 9]{W} and \cite{CM}. As a concrete example we will work out the structure theory in detail for $G=\textup{SU}(r,p)$ in Section \ref{SUrp}.

We fix from now on a real connected semisimple Lie group $G$ with finite center and Lie algebra $\mathfrak{g}_{\mathbb{R}}$.
Fix a Cartan involution $\theta_{\mathbb{R}}\in\textup{Aut}(\mathfrak{g}_{\mathbb{R}})$ and write 
\[
\mathfrak{g}_{\mathbb{R}}=\mathfrak{k}_{\mathbb{R}}\oplus\mathfrak{p}_{\mathbb{R}}
\]
for the resulting Cartan decomposition, with $\mathfrak{k}_{\mathbb{R}}=\mathfrak{g}_{\mathbb{R}}^{\theta_{\mathbb{R}}}$ the fix-point Lie subalgebra with respect to $\theta_{\mathbb{R}}$ and $\mathfrak{p}_{\mathbb{R}}$ the $(-1)$-eigenspace of $\theta_{\mathbb{R}}$.  
We fix the scalar product
\[
(x,y):=-K_{\mathfrak{g}_{\mathbb{R}}}(x,\theta_{\mathbb{R}}(y)),\qquad x,y\in\mathfrak{g}_{\mathbb{R}}
\]
on $\mathfrak{g}_{\mathbb{R}}$, with $K_{\mathfrak{g}_{\mathbb{R}}}$ the Killing form of $\mathfrak{g}_{\mathbb{R}}$.

Choose a maximal abelian subspace $\mathfrak{a}_{\mathbb{R}}$ of $\mathfrak{p}_{\mathbb{R}}$, which we view as Euclidean space with scalar product inherited from $(\cdot,\cdot)$. We endow the
linear dual $\mathfrak{a}_{\mathbb{R}}^*$ with the scalar product
such that the linear isomorphism $\mathfrak{a}_{\mathbb{R}}\overset{\sim}{\longrightarrow}
\mathfrak{a}_{\mathbb{R}}^*$, $h\mapsto (h,\cdot)$, becomes an isomorphism of Euclidean spaces.
We denote the scalar product on $\mathfrak{a}_{\mathbb{R}}^*$ again by $(\cdot,\cdot)$. 

For $\lambda\in\mathfrak{a}_{\mathbb{R}}^*$ set
\begin{equation*}
\begin{split}
\mathfrak{g}_{\mathbb{R}}^\lambda&:=\{x\in\mathfrak{g}_{\mathbb{R}}\,\, | \,\, [h,x]=\lambda(h)x\quad \forall\, h\in\mathfrak{a}_{\mathbb{R}}\},\\
\Sigma&:=\{\lambda\in\mathfrak{a}_{\mathbb{R}}^*\setminus\{0\}\,\, | \,\, \mathfrak{g}_{\mathbb{R}}^\lambda\not=\{0\}\}.
\end{split}
\end{equation*}
Then $\Sigma\subset \mathfrak{a}_{\mathbb{R}}^*$ is a (possibly non-reduced) root system, called the restricted root system. Furthermore,
\begin{equation}\label{rrsd}
\mathfrak{g}_{\mathbb{R}}=\mathfrak{g}_{\mathbb{R}}^0\oplus\bigoplus_{\lambda\in\Sigma}\mathfrak{g}_{\mathbb{R}}^\lambda
\end{equation}
and $\mathfrak{g}_{\mathbb{R}}^0=Z_{\mathfrak{g}_{\mathbb{R}}}(\mathfrak{a}_{\mathbb{R}})=\mathfrak{m}_{\mathbb{R}}\oplus\mathfrak{a}_{\mathbb{R}}$,
with $\mathfrak{m}_{\mathbb{R}}:=Z_{\mathfrak{k}_{\mathbb{R}}}(\mathfrak{a}_{\mathbb{R}})$ the centraliser of $\mathfrak{a}_{\mathbb{R}}$ in $\mathfrak{k}_{\mathbb{R}}$. 
The decomposition \eqref{rrsd} is orthogonal with respect to $\bigl(\cdot,\cdot\bigr)$,
and $\theta(\mathfrak{g}_{\mathbb{R}}^\lambda)=\mathfrak{g}_{\mathbb{R}}^{-\lambda}$.

Fix a maximal abelian subspace $\mathfrak{t}_{\mathbb{R}}$ of $\mathfrak{m}_{\mathbb{R}}$. Then 
\[
\mathfrak{h}_{\mathbb{R}}:=\mathfrak{t}_{\mathbb{R}}\oplus\mathfrak{a}_{\mathbb{R}}
\]
is a maximally noncompact $\theta_{\mathbb{R}}$-stable Cartan subalgebra of $\mathfrak{g}_{\mathbb{R}}$. Write $\mathfrak{g},\mathfrak{g}^\lambda,\mathfrak{k},\mathfrak{m},\mathfrak{h},\mathfrak{a},\mathfrak{t},...$ for the complexifications of the Lie algebras
$\mathfrak{g}_{\mathbb{R}},\mathfrak{g}_{\mathbb{R}}^{\lambda},\mathfrak{k}_{\mathbb{R}},\mathfrak{m}_{\mathbb{R}},\mathfrak{h}_{\mathbb{R}},\mathfrak{a}_{\mathbb{R}},\mathfrak{t}_{\mathbb{R}},....$.
The root space decomposition of $\mathfrak{g}$ with respect to $\mathfrak{h}$ is denoted by
\[
\mathfrak{g}=\mathfrak{h}\oplus\bigoplus_{\alpha\in R}\mathfrak{g}_\alpha
\]
with 
\begin{equation*}
\mathfrak{g}_\beta:=\{x\in\mathfrak{g} \,\, | \,\, [h,x]=\beta(h)x\quad \forall\, h\in\mathfrak{h}\}
\qquad (\beta\in\mathfrak{h}^*)
\end{equation*}
and root system $R:=\{\alpha\in\mathfrak{h}^*\setminus\{0\} \,\, | \,\, \mathfrak{g}_\alpha\not=\{0\}\}$. We write $\theta\in\textup{Aut}(\mathfrak{g})$ for the complex linear
extension of $\theta_{\mathbb{R}}$ to an involution of $\mathfrak{g}$. Then $\mathfrak{g}=\mathfrak{k}\oplus\mathfrak{p}$ is the decomposition of $\mathfrak{g}$ in $(+1)$--and $(-1)$-eigenspaces of $\theta$.
The involution $\theta$ restricts to an automorphism of $\mathfrak{h}$. The transpose of
$\theta|_{\mathfrak{h}}$ is a linear involution of $\mathfrak{h}^*$, which we will also denote by $\theta$. It
restricts to an involution of $R$ satisfying ${}\theta\alpha|_{\mathfrak{a}_{\mathbb{R}}}=
-\alpha|_{\mathfrak{a}_{\mathbb{R}}}$ for all $\alpha\in R$. We then have
$\theta(\mathfrak{g}_\alpha)=\mathfrak{g}_{\theta\alpha}$.

For $\lambda\in\Sigma\cup\{0\}$ we write
\[
R_\lambda:=\{\alpha\in R \,\,\, | \,\,\, \alpha|_{\mathfrak{a}_{\mathbb{R}}}=\lambda\}.
\]
Roots in $R_0$ are called imaginary roots (they are the roots in $R$ that take 
purely imaginary values on $\mathfrak{h}_{\mathbb{R}}$).
Note that $\theta$ fixes $R_0$ point-wise and maps $R_\lambda$ to $R_{-\lambda}$ 
for $\lambda\in\Sigma$. Furthermore,
\begin{equation*}
\mathfrak{m}=\mathfrak{t}\oplus\bigoplus_{\alpha\in R_0}\mathfrak{g}_\alpha,\qquad
\mathfrak{g}^\lambda=\bigoplus_{\alpha\in R_\lambda}\mathfrak{g}_\alpha\quad (\lambda\in\Sigma).
\end{equation*}
In particular, $R$ is the disjoint union of the $R_\lambda$ ($\lambda\in\Sigma\cup\{0\}$), $R_\lambda$ is nonempty for $\lambda\in\Sigma$, and the restriction map
\begin{equation}\label{rootrestrictionmap}
R\setminus R_0\twoheadrightarrow \Sigma,\quad
\alpha\mapsto \alpha|_{\mathfrak{a}_{\mathbb{R}}}
\end{equation}
is surjective. The strictly positive numbers
\[
\textup{mtp}(\lambda):=\#R_\lambda,\qquad \lambda\in\Sigma
\]
are called the restricted root multiplicities.
The Lie subalgebra $\mathfrak{m}$ of $\mathfrak{g}$ is reductive. The root system of the semisimple part $[\mathfrak{m},\mathfrak{m}]$ of $\mathfrak{m}$ with respect to the Cartan subalgebra $\mathfrak{t}_{\textup{ss}}:=\mathfrak{t}\cap [\mathfrak{m},\mathfrak{m}]$ is $\{\alpha|_{\mathfrak{t}_{\textup{ss}}}\}_{\alpha\in R_0}$.

Choose a decomposition $R=R^+\cup R^-$ of the root system $R$ into positive and negative roots such that 
\begin{enumerate}
\item $R_0=R_0^+\cup R_0^-$ with $R_0^{\pm}:=
(R^\pm\cap R_0)$ is a decomposition of $R_0$ in positive and negative roots,
\item $\Sigma=\Sigma^+\cup\Sigma^-$ with $\Sigma^\pm:=\{\alpha|_{\mathfrak{a}_{\mathbb{R}}} \,\, | \,\,
\alpha\in R^{\pm}\setminus R_0^{\pm}\}$ is a decomposition of $\Sigma$ in positive and negative roots
\end{enumerate}
(see, e.g., \cite[11.1.16]{Di}).
\begin{rema}
Restricted root systems appear more generally for so-called normal $\sigma$-systems of roots, see, e.g., \cite[\S 1.1.3]{W0}. A normal $\sigma$-system of roots consists of a root system $R$ and an involutive isometry $\sigma$ of the ambient Euclidean space, satisfying the additional properties that $\sigma$ stabilises $R$ and $\sigma\alpha-\alpha\not\in R$ for all $\alpha\in R$.

In our current setup, $R$ is a normal $(-\theta)$-system of roots. 
Indeed, 
suppose that $\alpha\in R$ satisfies $\alpha+\theta\alpha\in R$. Then 
$[e_\alpha,\theta(e_\alpha)]\in\mathfrak{p}\setminus\{0\}$.
On the other hand, $\alpha+\theta\alpha\in R_0$ since the root $\alpha+\theta\alpha$ vanishes on $\mathfrak{a}_{\mathbb{R}}$, and hence $[e_\alpha,\theta(e_\alpha)]\in\mathfrak{k}$. This is a contradiction. 
\end{rema}

Consider the nilpotent Lie subalgebra
\[
\mathfrak{n}_{\mathbb{R}}:=\bigoplus_{\lambda\in\Sigma^+}\mathfrak{g}^\lambda_{\mathbb{R}}
\]
 of $\mathfrak{g}_{\mathbb{R}}$. Its complexification within $\mathfrak{g}$
 is the nilpotent Lie subalgebra $\mathfrak{n}=\bigoplus_{\alpha\in R^+\setminus R_0^+}
 \mathfrak{g}_\alpha$. 
We then have the direct sum decompositions 
\[
\mathfrak{g}_{\mathbb{R}}=\mathfrak{k}_{\mathbb{R}}\oplus\mathfrak{a}_{\mathbb{R}}\oplus
\mathfrak{n}_{\mathbb{R}},\qquad \mathfrak{g}=\mathfrak{k}\oplus\mathfrak{a}\oplus \mathfrak{n}
\]
as vector spaces (infinitesimal Iwasawa decomposition).

Let $K,A,N\subset G$ be the connected Lie subgroups with Lie algebra $\mathfrak{k}_{\mathbb{R}}$, $\mathfrak{a}_{\mathbb{R}}$ and $\mathfrak{n}_{\mathbb{R}}$, respectively. Then $K\subset G$ is maximal compact, the multiplication map
\[
K\times A\times N\rightarrow G,\qquad (k,a,n)\mapsto kan
\]
is a diffeomorphism (the Iwasawa decomposition), and the exponential map $\exp$ of $G$
restricts to an isomorphism $\exp: \mathfrak{a}_{\mathbb{R}}\overset{\sim}{\longrightarrow} A$ of Lie groups. We write $\log: A\rightarrow \mathfrak{a}_{\mathbb{R}}$ for its inverse.

Write $M:=Z_K(A)$ for the centraliser of $A$ in $K$. It is a closed, but not necessarily connected, Lie subgroup of $K$ with Lie algebra $\mathfrak{m}_{\mathbb{R}}$. Note that $M=Z_K(\mathfrak{a})$
(the centraliser of $\mathfrak{a}$ in $K$ with respect to the adjoint action). Note that $\theta_{\mathbb{R}}$ and $\theta$ are $\textup{Ad}(K)$-linear, and
$\mathfrak{g}^\lambda$ is $\textup{Ad}(M)$-stable ($\lambda\in\Sigma\cup\{0\}$). 
In particular, $\mathfrak{n}_{\mathbb{R}}$ and $\mathfrak{n}$ are $\textup{Ad}(M)$-stable. 

The analytic Weyl group of $G$ is $W:=N_K(A)/M$, with $N_K(A)$ the normaliser of $A$ in $K$.
It acts naturally on $A$ and $\mathfrak{a}_{\mathbb{R}}$. The resulting contragredient $W$-action on $\mathfrak{a}_{\mathbb{R}}^*$ restricts to an action on the set $\Sigma$ of restricted roots. 

Denote by $W\backslash A$ the set of $W$-orbits in $A$, and write $K\backslash G/K$ for the double $(K,K)$-coset space of $G$. 
The map
\[
W\backslash A\overset{\sim}{\longrightarrow} K\backslash G/K,\qquad Wa\mapsto KaK
\]
is a bijection (see, e.g., \cite[\S VII.3]{Kn}). The decomposition $G=KAK$ is sometimes called the Cartan decomposition of $G$ with respect to $K$.

For $\nu\in\mathfrak{h}^*$ write $h_\nu\in\mathfrak{h}$ for the unique element such that
$K_{\mathfrak{g}}(h,h_\nu)=\nu(h)$ for all $h\in\mathfrak{h}$, with $K_{\mathfrak{g}}$ the Killing form of $\mathfrak{g}$. Choose root vectors $e_\alpha\in\mathfrak{g}_\alpha$  
such that $[e_\alpha,e_{-\alpha}]=h_\alpha$ ($\alpha\in R$).
For $\alpha\in R$ we have
\[
\theta(e_\alpha)=c_\alpha e_{\theta\alpha}
\]
for a unique nonzero scalar $c_\alpha\in\mathbb{C}^*$. Note that $c_\alpha=1$ for $\alpha\in R_0$.
Define for $\alpha\in R\setminus R_0$,
\[
y_\alpha:=e_\alpha+\theta(e_\alpha)=e_\alpha+c_\alpha e_{\theta\alpha}\in\mathfrak{k}.
\]

\begin{lem}\label{clem}
For $\alpha\in R\setminus R_0$ we have $c_{\theta\alpha}=c_\alpha^{-1}=c_{-\alpha}$. Furthermore,
$y_\alpha=c_\alpha y_{\theta\alpha}$.
\end{lem}
\begin{proof}
We have $e_\alpha=\theta^2(e_\alpha)=c_\alpha c_{\theta\alpha}e_\alpha$, hence 
$c_{\theta\alpha}=c_\alpha^{-1}$. 
Furthermore,
\[
\theta(h_\alpha)=\theta([e_\alpha,e_{-\alpha}])=[\theta(e_\alpha),\theta(e_{-\alpha})]=
c_\alpha c_{-\alpha}[e_{\theta\alpha},e_{-\theta\alpha}]=c_\alpha c_{-\alpha}
h_{\theta\alpha}.
\]
But $\theta(h_\alpha)=h_{\theta\alpha}$ since $K_{\mathfrak{g}}(\theta(x),\theta(y))=K_{\mathfrak{g}}(x,y)$ for $x,y\in\mathfrak{g}$, hence $c_{-\alpha}=c_\alpha^{-1}$.
Finally,
\[
y_\alpha=e_\alpha+\theta(e_\alpha)=\theta^2(e_\alpha)+c_\alpha e_{\theta\alpha}=
c_\alpha (e_{\theta\alpha}+\theta(e_{\theta\alpha}))=c_\alpha y_{\theta\alpha}.
\]
\end{proof}
Consider the direct sum decomposition
\[
\mathfrak{k}=\mathfrak{m}\oplus\mathfrak{q},\qquad \mathfrak{q}:=\bigoplus_{\alpha\in R^+\setminus R^+_0}\mathbb{C}y_\alpha
\]
as vector spaces.
Note that $\mathfrak{q}$ is the complex linear span of the orthogonal complement $\mathfrak{q}_{\mathbb{R}}$ of $\mathfrak{m}_{\mathbb{R}}$ in $\mathfrak{k}_{\mathbb{R}}$ with respect to $\bigl(\cdot,\cdot\bigr)$. Another description of $\mathfrak{q}$ is as follows.
Let $\textup{pr}_{\mathfrak{k}}: \mathfrak{g}\twoheadrightarrow\mathfrak{k}$ and 
$\textup{pr}_{\mathfrak{p}}: \mathfrak{g}\twoheadrightarrow\mathfrak{p}$ be the projections
along the direct sum decomposition $\mathfrak{g}=\mathfrak{k}\oplus\mathfrak{p}$. Then
$\textup{pr}_{\mathfrak{k}}=\frac{1}{2}(\textup{Id}_{\mathfrak{g}}+\theta)$ and $\textup{pr}_{\mathfrak{p}}=\frac{1}{2}(\textup{Id}_{\mathfrak{g}}-\theta)$, both projections are
$\textup{Ad}(K)$-equivariant, and $\mathfrak{q}$ is the image of $\mathfrak{n}$ under the projection $\textup{pr}_{\mathfrak{k}}$. In particular, $\mathfrak{q}$ is $\textup{Ad}(M)$-stable.

For $\lambda\in\mathfrak{a}^*$ 
the map $\xi_\lambda: A\rightarrow\mathbb{C}^*$, $a\mapsto a^\lambda:=e^{\lambda(\log a)}$, defines a complex-valued multiplicative character of $A$. Note that $\xi_\lambda\xi_\nu=\xi_{\lambda+\nu}$ 
($\lambda,\nu\in\mathfrak{a}^*$) and $\xi_0\equiv 1$. 
Furthermore, for $\lambda\in\Sigma\cup \{0\}$ we have
\[
\textup{Ad}_a(x)=a^\lambda x\qquad (x\in \mathfrak{g}^\lambda, a\in A).
\]
Write $A_{\textup{reg}}:=\{a\in A \,\,\, | \,\,\, a^\lambda\not=1\,\,\,\,\, \forall\, \lambda\in\Sigma\}$. 
For $a\in A_{\textup{reg}}$ and $\alpha\in R_\lambda$ ($\lambda\in\Sigma$) a direct computation shows that 
\begin{equation}\label{edecom}
\begin{split}
e_\alpha&=\frac{a^\lambda\textup{Ad}_{a^{-1}}(y_\alpha)-a^{2\lambda}y_\alpha}{1-a^{2\lambda}},\\
\theta(e_\alpha)&=\frac{a^\lambda\textup{Ad}_{a^{-1}}(y_\alpha)-y_\alpha}{a^{2\lambda}-1}.
\end{split}
\end{equation}
This leads to the direct sum decomposition
\begin{equation}\label{infinitesimalKAK}
\mathfrak{g}=\mathfrak{a}\oplus\textup{Ad}_{a^{-1}}(\mathfrak{q})\oplus\mathfrak{k}
\qquad (a\in A_{\textup{reg}})
\end{equation}
as vector spaces. It is this infinitesimal analogue of the Cartan decomposition that plays a crucial role in computing radial components of invariant differential operators, see \cite{HC,CM} and Section \ref{S4}.
\section{Generalised radial component maps}\label{S4}
\subsection{Differential operators}
The conjugation action of $N_K(A)$ on $A$ descends to an action of the analytic Weyl group $W=N_K(A)/M$. Its contragredient action on $C^\infty(A)$ is explicitly given by $(wf)(a):=f(g^{-1}ag)$ for $w=gM\in W$, $f\in C^\infty(A)$ and
$a\in A$. Note that the $W$-action on $A$ preserves $A_{\textup{reg}}$, and hence $C^\infty(A_{\textup{reg}})$ admits the same $W$-module algebra structure.

Note furthermore that $N_K(A)$ acts canonically on $\mathfrak{a}^*$, and the action descends to an action of $W$. Then $w\xi_\lambda:=\xi_{w\lambda}$ for
$w\in W$ and $\lambda\in\mathfrak{a}^*$. 
\begin{defi}\label{defiR}
We write $\mathcal{R}$ for the unital subalgebra of $C^\infty(A_{\textup{reg}})$ generated by $(1\pm \xi_{\lambda})^{-1}$ for $\lambda\in\Sigma$. 
\end{defi}
Note that $\mathcal{R}$ is a $W$-module subalgebra of $C^\infty(A_{\textup{reg}})$.

The action of $N_K(A)$ and $W$ on $C^\infty(A)$ has the following twisted vector-valued analogue.
Let $(\sigma,V)$ be a finite-dimensional $N_K(A)$-representation. The space $C^\infty(A;V)$ of smooth $V$-valued functions on $A$
is a left $N_K(A)$-module with action
\begin{equation}\label{fNKA}
(g\cdot f)(a):=\sigma(g)f(g^{-1}ag)\qquad (g\in N_K(A))
\end{equation}
for $f\in C^\infty(A;V)$ and $a\in A$.
Let $V^M\subseteq V$ be the subspace of $M$-invariant vectors in $V$. It is a $N_K(A)$-subrepresentation of $V$, and hence $C^\infty(A;V^M)$ is a $N_K(A)$-submodule of $C^\infty(A;V)$. The action of $N_K(A)$ on $C^\infty(A;V^M)$ descends
to an action of $W$. We write
$C^\infty(A;V^M)^W$ for the associated subspace of $W$-invariant functions. 

The results from the previous paragraph also hold true for smooth $V$-valued functions on $A_{\textup{reg}}$. Then $\mathcal{R}\otimes V$ is a $N_K(A)$-submodule of $C^\infty(A_{\textup{reg}};V)$, and $\mathcal{R}\otimes V^M$ is a $W$-submodule of 
$C^\infty(A_{\textup{reg}};V^M)$.

We write $h\mapsto \partial(h)$ ($h\in U(\mathfrak{a})$) for the action of $U(\mathfrak{a})$
on $C^\infty(A)$ by constant coefficient differential operators. Concretely, for $h\in\mathfrak{a}_{\mathbb{R}}$ the differential operator $\partial(h)$ is the directional derivative $\partial_h$ in direction $h$, 
\[
\bigl(\partial_hf\bigr)(a):=\frac{d}{dt}\biggr\rvert_{t=0}f(a\exp(th))\qquad
(f\in C^\infty(A),\,\, a\in A).
\]
We fix an orthonormal basis $\{x_1,\ldots,x_r\}$
of $\mathfrak{a}_{\mathbb{R}}$ with respect to $(\cdot,\cdot)$ once and for all, and
we write 
\[
\partial^{\mathbf{m}}:=\partial(x_1^{m_1}\cdots x_r^{m_r})=
\partial_{x_1}^{m_1}\cdots
\partial_{x_r}^{m_r}
\]
for $\mathbf{m}=(m_1,\ldots,m_r)\in\mathbb{Z}_{\geq 0}^r$. 

We now take $V=\textup{Hom}(E,F)$ with $E$ and $F$ two finite-dimensional $N_K(A)$-re\-pre\-sen\-ta\-tions, viewed as $N_K(A)$-representation with respect to the conjugation action. Then $V$ is canonically isomorphic to the tensor product $N_K(A)$-representation $F\otimes E^*$.
\begin{defi}
We write $\mathbb{D}(E,F)$ for the $N_K(A)$-module consisting of differential operators
$D=\sum_{\mathbf{m}}q_{\mathbf{m}}\partial^{\mathbf{m}}$ on $A_{\textup{reg}}$ with coefficients
$q_{\mathbf{m}}\in\mathcal{R}\otimes\textup{Hom}(E,F)$. 
\end{defi}
Concretely, the $N_K(A)$-module structure on $\mathbb{D}(E,F)$ is given by
\begin{equation}\label{mDO}
g\cdot\Bigl(\sum_{i}q_{i}\partial(h_i)\Bigr)=\sum_i(g\cdot q_i)\partial(\textup{Ad}_g(h_i))
\end{equation}
for $g\in N_K(A)$, $q_i\in\mathcal{R}\otimes\textup{Hom}(E,F)$ and $h_i\in U(\mathfrak{a})$.
The $N_K(A)$-action on differential operators $D\in\mathbb{D}(E,F)$, viewed as linear operators $D: C^\infty(A_{\textup{reg}};E)\rightarrow C^\infty(A_{\textup{reg}};F)$, is compatible with the $N_K(A)$-actions on $C^\infty(A_{\textup{reg}};E)$ and
$C^\infty(A_{\textup{reg}};F)$, 
\[
g\cdot \bigl(Df\bigr)=(g\cdot D)(g\cdot f)
\]
for $g\in N_K(A)$, $D\in\mathbb{D}(E,F)$ and $f\in C^\infty(A_{\textup{reg}};E)$.

Note that the $M$-action on $\mathbb{D}(E,F)$ is the trivial extension to $\mathbb{D}(E,F)$ of the conjugation $M$-action on $\textup{Hom}(E,F)$.

\begin{defi}
We write $\mathbb{D}_M(E,F)$ for the subspace of $M$-invariant differential operators in $\mathbb{D}(E,F)$. Concretely, it is the $N_K(A)$-module of differential operators $D=\sum_{\mathbf{m}}q_{\mathbf{m}}\partial^{\mathbf{m}}$ with $q_{\mathbf{m}}\in\mathcal{R}\otimes\textup{Hom}_M(E,F)$.
\end{defi}
The $N_K(A)$-action on $\mathbb{D}_M(E,F)$ descends to a $W$-action. We write
$\mathbb{D}_M(E,F)^W$ for the subspace of $W$-invariant differential operators in 
$\mathbb{D}_M(E,F)$ (equivalently, it is the subspace of $N_K(A)$-invariant differential operators
in $\mathbb{D}(E,F)$).

If $E=F$ then we write $\mathbb{D}(E)$, $\mathbb{D}_M(E)$ and $\mathbb{D}_M(E)^W$
for the algebras $\mathbb{D}(E,E)$, $\mathbb{D}_M(E,E)$ and $\mathbb{D}_M(E,E)^W$, respectively. Furthermore, for $D=\sum_{\mathbf{m}}q_{\mathbf{m}}\partial^{\mathbf{m}}\in\mathbb{D}(E,F)$ we call 
\[
D_M:=\sum_{\mathbf{m}}q_{\mathbf{m}}(\cdot)|_{E^M}\partial^{\mathbf{m}}\in
\mathbb{D}(E^M,F)
\]
the $M$-restriction of $D$.
It satisfies
\begin{equation}\label{redequation}
Df=D_Mf\qquad (D\in\mathbb{D}(E,F),\, f\in C^\infty(A_{\textup{reg}};E^M)).
\end{equation}
Note that the assignment $\mathbb{D}(E,F)\rightarrow \mathbb{D}(E^M,F)$, $D\mapsto D_M$ restricts to a linear map
\[
\mathbb{D}_M(E,F)\rightarrow \mathbb{D}(E^M,F^M),
\]
which is an algebra map $\mathbb{D}_M(E)\rightarrow\mathbb{D}(E^M)$ when $F=E$. 

Typically $E$ and $F$ are going to be $N_K(A)$-subrepresentations of the tensor product representation $V_\ell\otimes V_r^*$, with $(\sigma_\ell,V_\ell)$ and $(\sigma_r,V_r)$ finite-dimensional $K$-representations.

\subsection{Radial component maps}
This section recalls some well-known facts about Harish-Chandra's radial component maps. We follow closely the reference \cite{CM}.
By \cite[Thm. 2.4]{CM}
there exist for $u\in U(\mathfrak{g})$ unique elements
\begin{equation}\label{decompPi}
\begin{split}
\Pi(u)&=\sum_if_i\otimes h_i\otimes (u_i\otimes_{U(\mathfrak{m})}v_i)\in \mathcal{R}\otimes U(\mathfrak{a})\otimes
(U(\mathfrak{k})\otimes_{U(\mathfrak{m})}U(\mathfrak{k})),\\
\widetilde{\Pi}(u)&=\sum_i\widetilde{f}_i\otimes \widetilde{h}_i\otimes (\widetilde{u}_i\otimes_{U(\mathfrak{m})}\widetilde{v}_i)\in \mathcal{R}\otimes U(\mathfrak{a})\otimes
(U(\mathfrak{k})\otimes_{U(\mathfrak{m})}U(\mathfrak{k})),
\end{split}
\end{equation}
such that 
\begin{equation}\label{iCD}
\sum_if_i(a)\textup{Ad}_{a^{-1}}(u_i)h_iv_i=u=
\sum_i\widetilde{f}_i(a)\widetilde{u}_i\widetilde{h}_i\textup{Ad}_a(\widetilde{v}_i)\qquad \forall\, a\in A_{\textup{reg}}.
\end{equation}
The two decompositions \eqref{iCD} of $u\in U(\mathfrak{g})$ are called infinitesimal Cartan decompositions of $u$ relative to $a\in A_{\textup{reg}}$.
The maps $\Pi, \widetilde{\Pi}: U(\mathfrak{g})\rightarrow \mathcal{R}\otimes U(\mathfrak{a})
\otimes (U(\mathfrak{k})\otimes_{U(\mathfrak{m})}U(\mathfrak{k}))$ are called radial component maps. They provide the algebraic description of the radial components of the action of $u\in U(\mathfrak{g})$
as left and right $G$-invariant differential operators on $C^\infty(G)$ (see \cite{HC,CM} and Theorem \ref{DOwithrad}). 

\begin{rema}
Our present choice of coefficient algebra $\mathcal{R}$ (see Definition \ref{defiR}), which is different from the one used in \cite{CM}, is more convenient for our purposes since 
$\mathcal{R}\subset C^\infty(A_{\textup{reg}})$ is $W$-stable.
The results in \cite[\S2--3]{CM} are also valid in the present setup, with the same proofs.
\end{rema}

\begin{eg}
For $\alpha\in R_\lambda$ with $\lambda\in\Sigma$,
\begin{equation}\label{Pie}
\begin{split}
\Pi(e_\alpha)&=\frac{\xi_\lambda}{1-\xi_{2\lambda}}\otimes 1\otimes (y_\alpha\otimes_{U(\mathfrak{m})}1)-
\frac{\xi_{2\lambda}}{1-\xi_{2\lambda}}\otimes 1\otimes (1\otimes_{U(\mathfrak{m})}y_\alpha),\\
\widetilde{\Pi}(e_\alpha)&=\frac{1}{1-\xi_{2\lambda}}\otimes 1\otimes (y_\alpha\otimes_{U(\mathfrak{m})}1)-
\frac{\xi_{\lambda}}{1-\xi_{2\lambda}}\otimes 1\otimes (1\otimes_{U(\mathfrak{m})}y_\alpha)
\end{split}
\end{equation}
in view of the first line of \eqref{edecom}. 
\end{eg}
Note that
\begin{equation}\label{equal}
\Pi(u)=\widetilde{\Pi}(u)\qquad \forall\, u\in U(\mathfrak{g})^A,
\end{equation}
with $U(\mathfrak{g})^A\subseteq U(\mathfrak{g})$ the subalgebra of $\textup{Ad}(A)$-invariant elements in $U(\mathfrak{g})$.

Fix two finite-dimensional $K$-representations $(\sigma_\ell,V_\ell)$
and $(\sigma_r,V_r)$. Write $(V_\ell\otimes V_r^*)^{\mathfrak{m}}$ for the subspace
of $\mathfrak{m}$-invariant elements in the tensor product $\mathfrak{k}$-module $V_\ell\otimes V_r^*$, i.e.,
\[
\bigl(V_\ell\otimes V_r^*\bigr)^{\mathfrak{m}}:=\{T\in V_\ell\otimes V_r^* \,\, | \,\,
(\sigma_\ell(y)\otimes\textup{id}_{V_r^*})T+(\textup{id}_{V_\ell}\otimes\sigma_r^*(y))T=0
\quad \forall\, y\in\mathfrak{m}\}.
\]
Under the natural vector space
identification $V_\ell\otimes V_r^*\simeq\textup{Hom}(V_r,V_\ell)$, the invariant subspace
$(V_\ell\otimes V_r^*)^{\mathfrak{m}}$ corresponds to the subspace $\textup{Hom}_{\mathfrak{m}}(V_r,V_\ell)$ of $\mathfrak{m}$-intertwiners $V_r\rightarrow V_\ell$. Note that 
$(V_\ell\otimes V_r^*)^{\mathfrak{m}}$ is an $N_K(A)$-subrepresentation of $V_\ell\otimes V_r^*$
containing $(V_\ell\otimes V_r^*)^M$. It may be a strict inclusion since $M$ is not necessarily connected. 

Consider now the linear map 
\[
\zeta_{\sigma_\ell,\sigma_r}: U(\mathfrak{k})\otimes_{U(\mathfrak{m})}U(\mathfrak{k})\rightarrow
\textup{Hom}\bigl((V_\ell\otimes V_r^*)^{\mathfrak{m}},V_\ell\otimes V_r^*)
\]
defined by 
\[
\zeta_{\sigma_\ell,\sigma_r}(u\otimes_{U(\mathfrak{m})} v)T:=(\sigma_\ell(u)\otimes\sigma_r^*(S(v)))T
\qquad (u,v\in U(\mathfrak{k}), T\in (V_\ell\otimes V_r^*)^{\mathfrak{m}}).
\]
The map is clearly well-defined.
\begin{defi}\label{defDz}
Let 
\[
\mathcal{R}\otimes U(\mathfrak{a})\otimes (U(\mathfrak{k})\otimes_{U(\mathfrak{m})}U(\mathfrak{k}))
\rightarrow \mathbb{D}((V_\ell\otimes V_r^*)^{\mathfrak{m}}, V_\ell\otimes V_r^*),
\qquad z\mapsto L_z^{\sigma_\ell, \sigma_r}
\]
be the linear map such that
\[
L_z^{\sigma_\ell,\sigma_r}:=\zeta_{\sigma_\ell,\sigma_r}(u\otimes_{U(\mathfrak{m})} v)f\partial(h)
\]
for a pure tensor $z=f\otimes h\otimes (u\otimes_{U(\mathfrak{m})} v)\in\mathcal{R}\otimes U(\mathfrak{a})\otimes (U(\mathfrak{k})\otimes_{U(\mathfrak{m})}U(\mathfrak{k}))$.
\end{defi}

Write
$f\rvert_A\in C^\infty(A;V_\ell\otimes V_r^*)$ for the restriction 
of $f\in C^\infty(G;V_\ell\otimes V_r^*)$ 
to $A$.
Note that $\rvert_A$ restricts to an injective linear map
\[
\rvert_A: C_{\sigma_\ell,\sigma_r}^\infty(G)\hookrightarrow
C^\infty(A;(V_\ell\otimes V_r^*)^M)^W.
\]
Recall from Section \ref{S3} the actions $u[1]$ and $u\langle 1\rangle$ of $u\in U(\mathfrak{g})$
on 
$C^\infty(G;V)$ as $G$-invariant differential operators. Recall furthermore that
\[
L_{z;M}^{\sigma_\ell,\sigma_r}\in\mathbb{D}((V_\ell\otimes V_r^*)^M,V_\ell\otimes V_r^*)
\]
denotes the restriction of $L_z^{\sigma_\ell,\sigma_r}$ to $\mathbb{D}((V_\ell\otimes V_r^*)^M,
V_\ell\otimes V_r^*)$.
We now have the following theorem, which
is essentially \cite[Thm. 3.1]{CM}.
\begin{thm}\label{DOwithrad}
Let $u\in U(\mathfrak{g})$, $f\in C_{\sigma_\ell,\sigma_r}^\infty(G)$ and $a\in A_{\textup{reg}}$.
Then 
\begin{equation}\label{refDOrad}
\begin{split}
(u[1]f)(a)&=\bigl(L_{\Pi(u);M}^{\sigma_\ell,\sigma_r}f\rvert_A\bigr)(a),\\
(u\langle 1\rangle f)(a)&=\bigl(L_{\widetilde{\Pi}(S(u));M}^{\sigma_\ell,\sigma_r}f\rvert_A\bigr)(a).
\end{split}
\end{equation}
\end{thm}
\begin{proof}
By Remark \ref{useful}
the second equality of \eqref{refDOrad} is equivalent to
\begin{equation}\label{refDOradalt}
\bigl(\textup{Ad}_{a^{-1}}(u)[1]f\bigr)(a)=\bigl(L_{\widetilde{\Pi}(u);M}^{\sigma_\ell,\sigma_r}f\rvert_A\bigr)(a).
\end{equation}
By the definition of $\Pi(u)$ and $\widetilde{\Pi}(u)$ as well as the definition of $L_z^{\sigma_\ell,\sigma_r}$ (see Definition \ref{defDz}) it then suffices to prove that
\[
\bigl((\textup{Ad}_{a^{-1}}(u)hv)f\bigr)(a)=\zeta_{\sigma_\ell,\sigma_r}(u\otimes_{U(\mathfrak{m})} v)
\bigl(\partial(h)f\bigr)(a)
\]
for $u,v\in U(\mathfrak{k})$, $h\in U(\mathfrak{a})$ and $a\in A_{\textup{reg}}$. This is shown in the proof of \cite[Thm. 3.1]{CM}. 
\end{proof}
Note that 
\begin{equation}\label{lrsame}
L_{\widetilde{\Pi}(u)}^{\sigma_\ell,\sigma_r}=L_{\Pi(u)}^{\sigma_\ell,\sigma_r}
\qquad (u\in U(\mathfrak{g})^A)
\end{equation}
in view of \eqref{equal}.

The following proposition shows that the differential operators $L_{\Pi(u);M}^{\sigma_\ell,\sigma_r}$ and $L_{\widetilde{\Pi}(u);M}^{\sigma_\ell,\sigma_r}$ 
are determined by the property \eqref{refDOrad} for all $f\in C_{\sigma_\ell,\sigma_r}^\infty(G)$.
 
\begin{prop}\label{zeroproperty}
Suppose that $D\in\mathbb{D}((V_\ell\otimes V_r^*)^M,V_\ell\otimes V_r^*)$ satisfies
$D(f\rvert_A)=0$ for all $f\in C_{\sigma_\ell,\sigma_r}^{\infty}(G)$. Then $D=0$.
\end{prop}
\begin{proof}
This is shown as part of the proof of \cite[Thm. 3.3]{CM}.
\end{proof}
The radial components satisfy the following equivariance property.
\begin{lem}\label{invNKA}
For $u\in U(\mathfrak{g})$ and $g\in N_K(A)$ we have
\begin{equation}\label{Nequivariance2}
L_{\Pi(\textup{Ad}_g(u))}^{\sigma_\ell,\sigma_r}=g\cdot L_{\Pi(u)}^{\sigma_\ell,\sigma_r},\qquad
L_{\widetilde{\Pi}(\textup{Ad}_g(u))}^{\sigma_\ell,\sigma_r}=
g\cdot L_{\widetilde{\Pi}(u)}^{\sigma_\ell,\sigma_r}
\end{equation}
in $\mathbb{D}((V_\ell\otimes V_r^*)^{\mathfrak{m}},V_\ell\otimes V_r^*)$.
\end{lem}
\begin{proof}
We only prove the first equality, the second is proved in a similar manner.
Fix $g\in N_K(A)$ and write $w:=gM\in W$.  Using the notation \eqref{decompPi} we have for $a\in A_{\textup{reg}}$,
\begin{equation}\label{Nequivariance}
\textup{Ad}_g(u)=\sum_i(w\cdot f_i)(wa)\textup{Ad}_{(wa)^{-1}}(\textup{Ad}_g(u_i))\textup{Ad}_g(h_i)
\textup{Ad}_g(v_i).
\end{equation}
Since $wf_i\in\mathcal{R}$, $\textup{Ad}_g(h_i)\in U(\mathfrak{a})$ and $\textup{Ad}_g(u_i),\textup{Ad}_g(v_i)\in U(\mathfrak{k})$, formula \eqref{Nequivariance} is an infinitesimal Cartan decomposition of $\textup{Ad}_g(u)$ relative to
$wa\in A_{\textup{reg}}$ for each $a\in A_{\textup{reg}}$. Hence
\[
\Pi(\textup{Ad}_g(u))=\sum_iw\cdot f_i\otimes\textup{Ad}_g(h_i)\otimes (\textup{Ad}_g(u_i)\otimes_{U(\mathfrak{m})}\textup{Ad}_g(v_i)).
\]
The lemma then follows from the definition of $L_{z}^{\sigma_\ell,\sigma_r}$ (Definition \ref{defDz}) and the fact that
\[
\zeta_{\sigma_\ell,\sigma_r}(\textup{Ad}_g(u_i)\otimes_{U(\mathfrak{m})}\textup{Ad}_g(v_i))=
(\sigma_\ell(g)\otimes\sigma_r^*(g))\circ\zeta_{\sigma_\ell,\sigma_r}(u_i\otimes v_i)\circ
(\sigma_\ell(g^{-1})\otimes\sigma_r^*(g^{-1}))
\]
in $\textup{Hom}((V_\ell\otimes V_r^*)^{\mathfrak{m}},V_\ell\otimes V_r^*)$.
\end{proof} 
\subsection{Coordinate radial component maps}\label{Scrcm}
Fix finite-dimensional $G$-representations $(\tau_i,U_i)$ ($1\leq i\leq n$) and two finite-dimensional $K$-representations $(\sigma_\ell,V_\ell)$ and $(\sigma_r, V_r)$. Recall that $(\sigma_{\ell;n},V_\ell\otimes\underline{U})$ is the $K$-representation with $K$ acting diagonally on $V_\ell\otimes\underline{U}$.

\begin{defi}\label{Dcoor}
For $u\in U(\mathfrak{g})$ and $0\leq j\leq n$ define differential operators
\[
D_{u,j}^{\sigma_\ell,\underline{\tau},\sigma_r}\in
\mathbb{D}((V_\ell\otimes\underline{U}\otimes V_r^*)^{\mathfrak{m}},V_\ell\otimes\underline{U}\otimes V_r^*)
\]
by the formulas
\begin{equation*}
\begin{split}
D_{u,j}^{\sigma_\ell,\underline{\tau},\sigma_r}&:=\sum_{(u)}\tau_{j+1}(S(u_{(1)}))\cdots
\tau_n(S(u_{(n-j)}))L_{\widetilde{\Pi}(u_{(n-j+1)})}^{\sigma_{\ell;n},\sigma_r}\qquad
(0\leq j<n),\\
D_{u,n}^{\sigma_\ell,\underline{\tau},\sigma_r}&:=L_{\Pi(u)}^{\sigma_{\ell;n},\sigma_r}.
\end{split}
\end{equation*}
\end{defi}
The following lemma follows directly from Lemma \ref{invNKA}.
\begin{lem}\label{Mlem}
For $u\in U(\mathfrak{g})$, $0\leq j\leq n$ and $g\in N_K(A)$ we have
\[
D_{\textup{Ad}_g(u),j}^{\sigma_\ell,\underline{\tau},\sigma_r}=g\cdot D_{u,j}^{\sigma_\ell,\underline{\tau},\sigma_r}
\]
in $\mathbb{D}((V_\ell\otimes\underline{U}\otimes V_r^*)^{\mathfrak{m}},V_\ell\otimes\underline{U}\otimes V_r^*)$.
\end{lem}
Let $U(\mathfrak{g})^{\mathfrak{m}}$ be the centraliser of $\mathfrak{m}$ in $U(\mathfrak{g})$.
It contains $U(\mathfrak{g})^M$ as a subalgebra. The following is a direct consequence of Lemma \ref{Mlem}.
\begin{cor}\label{Mcor}
The differential operator $D_{u,j}^{\sigma_\ell,
\underline{\tau},\sigma_r}\in
\mathbb{D}((V_\ell\otimes\underline{U}\otimes V_r^*)^{\mathfrak{m}},
V_\ell\otimes\underline{U}\otimes V_r^*)$ lies in 
\begin{equation*}
\begin{split}
&\mathbb{D}((V_\ell\otimes\underline{U}\otimes V_r^*)^{\mathfrak{m}})\quad\quad\,\,\,\, \hbox{ if }\,\,\,
u\in U(\mathfrak{g})^{\mathfrak{m}},\\
&\mathbb{D}_M((V_\ell\otimes\underline{U}\otimes V_r^*)^{\mathfrak{m}})\quad\,\,\,\,\, \hbox{ if }\,\,\,
u\in U(\mathfrak{g})^M,\\
&\mathbb{D}_M((V_\ell\otimes\underline{U}\otimes V_r^*)^{\mathfrak{m}})^W\quad \hbox{ if }\,\,\,
u\in U(\mathfrak{g})^{N_K(A)}.
\end{split}
\end{equation*}
In particular, $D_{u,j;M}^{\sigma_\ell,\underline{\tau},\sigma_r}\in
\mathbb{D}((V_\ell\otimes\underline{U}\otimes V_r^*)^M,
V_\ell\otimes\underline{U}\otimes V_r^*)$ lies in
\begin{equation*}
\begin{split}
&\mathbb{D}((V_\ell\otimes\underline{U}\otimes V_r^*)^M)\quad\quad \hbox{ if }\,\,\, u\in 
U(\mathfrak{g})^M,\\
&\mathbb{D}((V_\ell\otimes\underline{U}\otimes V_r^*)^M)^W\quad\, \hbox{ if }\,\,\,
u\in U(\mathfrak{g})^{N_K(A)}.
\end{split}
\end{equation*}
\end{cor} 
Next we show that the differential operator
$D_{u,j}^{\sigma_\ell,\underline{\tau},\sigma_r}$ arises as the radial component of 
the action $u[j]$ on $C_{\sigma_\ell,\underline{\tau},\sigma_r}^\infty(G^{\times (n+1)})$. The relevant restriction map is defined as follows.
\begin{defi}
Define the linear map
\[
\textup{Res}^\flat: C^\infty(G^{\times (n+1)}; V_\ell\otimes\underline{U}\otimes V_r^*)
\rightarrow C^\infty(A;V_\ell\otimes\underline{U}\otimes V_r^*)
\]
by $\textup{Res}^\flat(f):=f^\flat\rvert_A$. Concretely, for $f\in C^\infty(G^{\times (n+1)}; V_\ell\otimes\underline{U}\otimes V_r^*)$ and $a\in A$,
\[
\textup{Res}^\flat(f)(a):=f(1,\ldots,1,a).
\]
\end{defi}
By Lemma \ref{Nto0} the restriction of $\textup{Res}^\flat$ to $C_{\sigma_\ell,\underline{\tau},\sigma_r}^\infty(G^{\times (n+1)})$ is an injective linear map
\[
\textup{Res}^\flat: C_{\sigma_\ell,\underline{\tau},\sigma_r}^\infty(G^{\times (n+1)})
\hookrightarrow C^\infty(A; (V_\ell\otimes \underline{U}\otimes V_r^*)^M)^W.
\]

Theorem \ref{DOwithrad} now generalises as follows.
\begin{thm}\label{DOwithradcomp}
For $u\in U(\mathfrak{g})$, $0\leq j\leq n$,
$f\in C_{\sigma_\ell,\underline{\tau},\sigma_r}^\infty(G^{\times (n+1)})$ and $a\in A_{\textup{reg}}$
we have
\begin{equation}\label{keyeqn}
\textup{Res}^\flat\bigl(u[j]f\bigr)(a)=\bigl(D_{u,j;M}^{\sigma_\ell,\underline{\tau},\sigma_r}\textup{Res}^\flat(f)\bigr)(a).
\end{equation}
\end{thm}
\begin{proof}
Let $u\in U(\mathfrak{g})$, $f\in C_{\sigma_\ell,\underline{\tau},\sigma_r}^\infty(G^{\times (n+1)})$ and $a\in A_{\textup{reg}}$. Then $f^\flat\in C_{\sigma_{\ell;n},\sigma_r}^\infty(G)$ by Lemma \ref{Nto0}, and we have
\[
\textup{Res}^\flat\bigl(u[n]f\bigr)(a)=\bigl(u[n]f\bigr)^\flat(a)=\bigl(u[1]f^\flat\bigr)(a)=\bigl(L_{\Pi(u);M}^{\sigma_{\ell;n},\sigma_r}
f^\flat\rvert_A\bigr)(a)=\bigl(D_{u,n:M}^{\sigma_\ell,\underline{\tau},\sigma_r}
\textup{Res}^\flat(f)\bigr)(a),
\]
where the third equality follows from the first line of \eqref{refDOrad}. This proves \eqref{keyeqn}
for $j=n$. 

Suppose that $0\leq j<n$. By Corollary \ref{corflat} we have
\[
\bigl(u[j]f\bigr)^\flat(a)
=\sum_{(u)}\tau_{j+1}(S(u_{(1)}))\cdots\tau_n(S(u_{(n-j)}))
\bigl(S(u_{(n-j+1)})\langle n\rangle f\bigr)^\flat(a).
\]
Formula \eqref{keyeqn} then follows from the fact that 
\[
\bigl(S(u_{(n-j+1)})\langle n\rangle f\bigr)^\flat(a)=
\bigl(S(u_{(n-j+1)})\langle 1\rangle f^\flat\bigr)(a)=
\bigl(L_{\widetilde{\Pi}(u_{(n-j+1)});M}^{\sigma_{\ell;n},\sigma_r}\textup{Res}^\flat(f)\bigr)(a),
\]
where the last equality is due to the second equation of \eqref{refDOrad}.
\end{proof}
\begin{rema}\label{lefttoright}
Replacing the role of left $G$-invariant differential operators by right $G$-invariant differential operators gives 
\[
\textup{Res}^\flat\bigl(u\langle j\rangle f\bigr)(a)=\bigl(\widehat{D}_{S(u),j;M}^{\sigma_\ell,
\underline{\tau},\sigma_r}\textup{Res}^\flat(f)\bigr)(a)
\]
for $u\in U(\mathfrak{g})$, $0\leq j\leq n$, $f\in C_{\sigma_\ell,\underline{\tau},\sigma_r}^\infty(G^{\times (n+1)})$ and $a\in A_{\textup{reg}}$, with the differential operators
$\widehat{D}_{u,j}^{\sigma_\ell,
\underline{\tau},\sigma_r}\in \mathbb{D}((V_\ell\otimes\underline{U}\otimes V_r^*)^{\mathfrak{m}},
V_\ell\otimes\underline{U}\otimes V_r^*)$ given by
\begin{equation}
\widehat{D}_{u,j}^{\sigma_\ell,
\underline{\tau},\sigma_r}:=
\sum_{(u)}\tau_{j+1}(S(u_{(n-j+1)}))\cdots\tau_n(S(u_{(2)}))
L_{\widetilde{\Pi}(u_{(1)})}^{\sigma_{\ell;n},\sigma_r}.
\end{equation}
Lemma \ref{Mlem} and Corollary \ref{Mcor} also apply for 
$\widehat{D}_{u,j}^{\sigma_\ell,\underline{\tau},\sigma_r}$.
\end{rema}
Define the linear map
\begin{equation}\label{globaldef}
\bigl(U(\mathfrak{g})^M\bigr)^{\otimes (n+1)}\rightarrow \mathbb{D}_M((V_\ell\otimes\underline{U}\otimes V_r^*)^{\mathfrak{m}}),\qquad
X\mapsto D_X^{\sigma_\ell,\underline{\tau},\sigma_r}
\end{equation}
by
\[
D_X^{\sigma_\ell,\underline{\tau},\sigma_r}:=
\widehat{D}_{u_0,0}^{\sigma_\ell,\underline{\tau},\sigma_r}\circ\Bigl(
D_{u_1,1}^{\sigma_\ell,\underline{\tau},\sigma_r}\circ\cdots\circ
D_{u_{n-1},n-1}^{\sigma_\ell,\underline{\tau},\sigma_r}\Bigr)\circ
D_{u_n,n}^{\sigma_\ell,\underline{\tau},\sigma_r}
\qquad (X=u_0\otimes\cdots\otimes u_n).
\]
Note that its restriction $D_{X;M}^{\sigma_\ell,\underline{\tau},\sigma_r}$ to $\mathbb{D}((V_\ell\otimes\underline{U}\otimes V_r^*)^M)$ is the composition of the restricted coordinate-wise radial components.
\begin{thm}\label{mainthm}
The assignment $X\mapsto D_{X;M}^{\sigma_\ell,\underline{\tau},\sigma_r}$ restricts to
an algebra homomorphism
\begin{equation*}
U(\mathfrak{g})^{K,\textup{opp}}\otimes Z(U(\mathfrak{g}))^{\otimes (n-1)}
\otimes U(\mathfrak{g})^K\rightarrow
\mathbb{D}((V_\ell\otimes \underline{U}\otimes V_r^*)^M)^W.
\end{equation*}
\end{thm}
\begin{proof}
Fix an element $X\in U(\mathfrak{g})^{K,\textup{opp}}\otimes Z(U(\mathfrak{g}))^{\otimes (n-1)}\otimes
U(\mathfrak{g})^K$ and $f\in C_{\sigma_\ell,\underline{\tau},\sigma_r}^\infty(G^{\times (n+1)})$.
We have $D_{X;M}^{\sigma_\ell,\underline{\tau},\sigma_r}\in\mathbb{D}((V_\ell\otimes\underline{U}\otimes V_r^*)^M)^W$ by Corollary \ref{Mcor} and Remark \ref{lefttoright}. Furthermore,
$X\ast f\in C_{\sigma_\ell,\underline{\tau},\sigma_r}^{\infty}(G^{\times (n+1)})$
by Lemma \ref{qiaction}. By Theorem \ref{DOwithradcomp}, Remark \ref{lefttoright} and \eqref{actionoverall} we then have
\[
\textup{Res}^\flat(X\ast f)=D_{X;M}^{\sigma_\ell,\underline{\tau},\sigma_r}\bigl(
\textup{Res}^\flat f\bigr).
\]
Choosing a second element 
$Y\in U(\mathfrak{g})^{K,\textup{opp}}\otimes Z(U(\mathfrak{g}))^{\otimes (n-1)}\otimes
U(\mathfrak{g})^K$ we get
\begin{equation}\label{hhhh}
D_{XY;M}^{\sigma_\ell,\underline{\tau},\sigma_r}\bigl(\textup{Res}^\flat f\bigr)=
\textup{Res}^\flat(X\ast(Y\ast f))=D_{X;M}^{\sigma_\ell,\underline{\tau},\sigma_r}
\bigl(D_{Y;M}^{\sigma_\ell,\underline{\tau},\sigma_r}\bigl(\textup{Res}^\flat f\bigr)\bigr).
\end{equation}
Lemma \ref{Nto0} and \eqref{hhhh} now imply that
\[
D_{XY;M}^{\sigma_\ell,\underline{\tau},\sigma_r}\bigl(F\rvert_A\bigr)=
\bigl(D_{X;M}^{\sigma_\ell,\underline{\tau},\sigma_r}\circ
D_{Y;M}^{\sigma_\ell,\underline{\tau},\sigma_r}\bigr)\bigl(F\rvert_A\bigr)
\qquad \forall\, F\in C_{\sigma_{\ell;n},\sigma_r}^\infty(G).
\]
applying Proposition \ref{zeroproperty} with respect to the two $K$-representations 
$(\sigma_{\ell;n},V_\ell\otimes\underline{U})$ and $(\sigma_r,V_r)$ we conclude that
\[
D_{XY;M}^{\sigma_\ell,\underline{\tau},\sigma_r}=
D_{X;M}^{\sigma_\ell,\underline{\tau},\sigma_r}\circ D_{Y;M}^{\sigma_\ell,\underline{\tau},\sigma_r}
\]
in $\mathbb{D}((V_\ell\otimes\underline{U}\otimes V_r^*)^M)^W$.
\end{proof}
\begin{rema}
For $n=0$, Theorem \ref{mainthm} should be read as follows: the assignment
\[
u\otimes_{Z(U(\mathfrak{g}))}v\mapsto L_{\widetilde{\Pi}(u);M}^{\sigma_\ell,\sigma_r}\circ
L_{\Pi(v);M}^{\sigma_\ell,\sigma_r}\qquad (u,v\in U(\mathfrak{g})^K)
\]
defines an algebra homomorphism
$U(\mathfrak{g})^{K,\textup{opp}}\otimes_{Z(U(\mathfrak{g}))}U(\mathfrak{g})^K\rightarrow
\mathbb{D}((V_\ell\otimes V_r^*)^M)^W$ (cf. \cite[Thm. 3.3]{CM}). The balancing condition of the tensor product over the center is justified by \eqref{lrsame}, see also Remark \ref{qiactionrem}.
\end{rema}
By Theorem \ref{mainthm} we have the inclusion of algebras
\begin{equation}\label{incalgebra}
A^{\sigma_\ell,\underline{\tau},\sigma_r}\subseteq
B^{\sigma_\ell,\underline{\tau},\sigma_r}\subseteq
\mathbb{D}((V_\ell\otimes\underline{U}\otimes V_r^*)^M)^W
\end{equation}
with 
\[
B^{\sigma_\ell,\underline{\tau},\sigma_r}:=\{D_{X;M}^{\sigma_\ell,\underline{\tau},\sigma_r}\,\, 
| \,\, X\in U(\mathfrak{g})^{K,\textup{opp}}\otimes Z(U(\mathfrak{g}))^{\otimes (n-1)}
\otimes U(\mathfrak{g})^K\}
\]
and $A^{\sigma_\ell,\underline{\tau},\sigma_r}$ the commutative subalgebra
\[
A^{\sigma_\ell,\underline{\tau},\sigma_r}:=
\{D_{X;M}^{\sigma_\ell,\underline{\tau},\sigma_r}\,\, 
| \,\, X\in Z(U(\mathfrak{g}))^{\otimes (n+1)}\}.
\]
Furthermore, we have
\[
D_{X;M}^{\sigma_\ell,\underline{\tau},\sigma_r}=
D_{u_0,0;M}^{\sigma_\ell,\underline{\tau},\sigma_r}\circ
\cdots\circ D_{u_{n-1},n-1;M}^{\sigma_\ell,\underline{\tau},\sigma_r}\circ
D_{u_n,n;M}^{\sigma_\ell,\underline{\tau},\sigma_r}
\qquad (X=u_0\otimes\cdots\otimes u_n\in Z(U(\mathfrak{g}))^{\otimes (n+1)})
\]
due to the following lemma.
\begin{lem}\label{samelemma}
For $u\in Z(U(\mathfrak{g}))$ and $0\leq j\leq n$ we have
\[
\widehat{D}_{u,j;M}^{\sigma_\ell,\underline{\tau},\sigma_r}=
D_{u,j;M}^{\sigma_\ell,\underline{\tau},\sigma_r}
\]
in $\mathbb{D}((V_\ell\otimes\underline{U}\otimes V_r^*)^M)^W$.
\end{lem}
\begin{proof}
Let $u\in Z(U(\mathfrak{g}))$ and $f\in C_{\sigma_\ell,\underline{\tau},\sigma_r}^\infty(G^{\times (n+1)})$. Then we have for $a\in A_{\textup{reg}}$,
\begin{equation*}
\begin{split}
\bigl(\widehat{D}_{u,j;M}^{\sigma_\ell,\underline{\tau},\sigma_r}\textup{Res}^\flat(f)\bigr)(a)&=
\textup{Res}^\flat\bigl(S(u)\langle j\rangle f\bigr)(a)\\
&=\textup{Res}^\flat(u[j]f\bigr)(a)=\bigl(D_{u,j;M}^{\sigma_\ell,\underline{\tau},\sigma_r}
\textup{Res}^\flat(f)\bigr)(a)
\end{split}
\end{equation*}
with the first equality by Remark \ref{lefttoright}, the second by Remark \ref{useful}, and the third by Theorem \ref{DOwithradcomp}. Now the conclusion follows by Lemma \ref{Nto0} and Proposition
\ref{zeroproperty}, cf. the proof of Theorem \ref{mainthm}.
\end{proof}

Suitable gauged versions of the subalgebras $B^{\sigma_\ell,\underline{\tau},\sigma_r}$
and $A^{\sigma_\ell,\underline{\tau},\sigma_r}$ will serve as the algebras of quantum integrals
and quantum Hamiltonians of a quantum superintegrable system, see Subsection \ref{Further}.

The following corollary of Theorem \ref{mainthm} and Lemma \ref{samelemma} is now immediate.
\begin{cor}\label{commcoord}
For $u,v\in Z(U(\mathfrak{g}))$ and $0\leq i,j\leq n$ we have
\[
\lbrack D_{u,i;M}^{\sigma_\ell,\underline{\tau},\sigma_r},
D_{v,j;M}^{\sigma_\ell,\underline{\tau},\sigma_r}\rbrack=0
\]
in $\mathbb{D}((V_\ell\otimes\underline{U}\otimes V_r^*)^M)^W$.
\end{cor}

\begin{rema}\label{SplitRemark2}
For finite-dimensional $\mathfrak{k}$-modules $(\sigma_\ell, V_\ell)$ and
$(\sigma_r,V_r)$ we expect that  
\[
\lbrack D_{u,i}^{\sigma_\ell,\underline{\tau},\sigma_r},
D_{v,j}^{\sigma_\ell,\underline{\tau},\sigma_r}\rbrack =0
\]
in $\mathbb{D}_M((V_\ell\otimes\underline{U}\otimes V_r^*)^{\mathfrak{m}})$ for $u,v\in Z(U(\mathfrak{g}))$ and $0\leq i,j\leq n$.
This follows from the theory of formal $n$-point spherical functions when $G$ is real split, 
see \cite{SR}. 
\end{rema}
\section{The quantum Calogero-Moser spin chain}\label{Further}
Fix throughout this section finite-dimensional smooth $K$-repesentations $(\sigma_\ell,V_\ell)$, $(\sigma_r,V_r)$ and finite-dimensional smooth $G$-representations $(\tau_i,U_i)$ ($1\leq i\leq n$).

\subsection{A dynamical factorisation of the Casimir element}
Denote the multiplication map of $U(\mathfrak{g})$ by $\mu\in\textup{Hom}(U(\mathfrak{g})\otimes U(\mathfrak{g}),U(\mathfrak{g}))$. 
\begin{defi}
The Casimir
element $\Omega\in Z(U(\mathfrak{g}))$ is 
\[
\Omega:=\mu(\varpi)
\]
with $\varpi\in S^2(\mathfrak{g})^{G}$ the 
$G$-invariant symmetric tensor  associated to $K_{\mathfrak{g}}$.
\end{defi}
Explicitly, for a basis $\{b_s\}_s$ of $\mathfrak{g}$ 
we have $\varpi=\sum_sb_s\otimes b_s^\prime$ and $\Omega=\sum_sb_sb_s^\prime$,
where $\{b_s^\prime\}_s$ is the basis of $\mathfrak{g}$ dual to $\{b_s\}_s$ with respect to
$K_{\mathfrak{g}}$.

In the computations we will use a basis of $\mathfrak{g}$ compatible with the root space decomposition of $\mathfrak{g}$. It contains a basis $\{z_j\}_{j=1}^{\textup{dim}(\mathfrak{h})}$ of $\mathfrak{h}$ such that
\begin{enumerate}
\item $K_{\mathfrak{g}}(z_j,z_{j^\prime})=\delta_{j,j^\prime}$,
\item $x_j:=z_j$ ($1\leq j\leq r$) forms a basis of $\mathfrak{a}_{\mathbb{R}}$,
\item $iz_j$ ($r+1\leq j\leq\textup{dim}(\mathfrak{h})$) forms a basis of $\mathfrak{t}_{\mathbb{R}}$.
\end{enumerate}
Recall that $e_\alpha\in\mathfrak{g}_\alpha$ are root vectors such that $[e_\alpha,e_{-\alpha}]=h_\alpha$. Then $\{z_j\}_{j=1}^{\textup{dim}(\mathfrak{h})}\cup \{e_{-\alpha}\}_{\alpha\in R}$ is the basis of $\mathfrak{g}$ dual to $\{z_j\}_{j=1}^{\textup{dim}(\mathfrak{h})}\cup\{e_\alpha\}_{\alpha\in R}$ with respect to $K_\mathfrak{g}$ (cf. \cite[Prop. 8.3\,c]{Hu}), hence
\begin{equation}\label{varpi}
\varpi=\sum_{j=1}^{\textup{dim}(\mathfrak{h})}z_j\otimes z_j+\sum_{\alpha\in R}e_{-\alpha}\otimes e_\alpha
\end{equation}
and
\begin{equation}\label{Omega}
\Omega=\sum_{j=1}^{\textup{dim}(\mathfrak{h})}z_j^2+\sum_{\alpha\in R}e_{-\alpha}e_{\alpha}.
\end{equation}
Note that $\varpi$ can be recovered from $\Omega$ by
\[
\varpi=\frac{1}{2}\bigl(\Delta(\Omega)-\Omega\otimes 1-1\otimes \Omega\bigr).
\]

Write $\Omega_{\mathfrak{m}}:=\mu(\varpi_{\mathfrak{m}})\in Z(U(\mathfrak{m}))$, with $\varpi_{\mathfrak{m}}\in S^2(\mathfrak{m})^M$ the symmetric tensor associated to 
the 
nondegenerate bilinear form $K_{\mathfrak{g}}(\cdot,\cdot)|_{\mathfrak{m}\times\mathfrak{m}}$,
\begin{equation}\label{varpim}
\varpi_{\mathfrak{m}}:=\sum_{j=r+1}^{\textup{dim}(\mathfrak{h})}z_j\otimes z_j+
\sum_{\alpha\in R_0}e_{-\alpha}\otimes e_\alpha.
\end{equation}

Finally, $\varpi^\prime:=\varpi-\varpi_{\mathfrak{m}}\in S^2(\mathfrak{g})^M$ and
$\Omega^\prime:=\mu(\varpi^\prime)\in U(\mathfrak{g})^M$ admit the explicit expressions
\begin{equation}\label{Omegam}
\varpi^\prime=\sum_{j=1}^rx_j\otimes x_j+\sum_{\lambda\in\Sigma}\varpi_\lambda,\qquad
\Omega^\prime=\sum_{j=1}^rx_j^2+\sum_{\lambda\in\Sigma}\Omega_\lambda
\end{equation}
with, for a restricted root $\lambda\in\Sigma$,
\begin{equation}\label{lambdaversions}
\varpi_\lambda:=\sum_{\alpha\in R_\lambda}e_{-\alpha}\otimes e_\alpha\in\mathfrak{g}^{-\lambda}\otimes\mathfrak{g}^\lambda,\qquad \Omega_\lambda:=\mu(\varpi_\lambda)=\sum_{\alpha\in R_\lambda}e_{-\alpha}e_\alpha\in U(\mathfrak{g}).
\end{equation}
\begin{lem}\label{Mfactors}
We have $\varpi_\lambda\in (\mathfrak{g}^{-\lambda}\otimes\mathfrak{g}^\lambda)^M$ and $\Omega_\lambda\in U(\mathfrak{g})^M$.
\end{lem}
\begin{proof}
The restricted root spaces $\mathfrak{g}^\lambda$ and $\mathfrak{g}^{-\lambda}$ are $\textup{Ad}(M)$-stable since $M=Z_K(\mathfrak{a})$. Let $g\in M$. The basis 
$\{\textup{Ad}_g(e_{-\alpha})\}_{\alpha\in R_\lambda}$ of $\mathfrak{g}^{-\lambda}$ is dual to the basis
$\{\textup{Ad}_g(e_\alpha)\}_{\alpha\in R_\lambda}$ of $\mathfrak{g}^\lambda$ with respect to the perfect $\textup{Ad}(M)$-invariant pairing
$K_{\mathfrak{g}}(\cdot,\cdot)|_{\mathfrak{g}^{-\lambda}\times\mathfrak{g}^{\lambda}}:
\mathfrak{g}^{-\lambda}\times\mathfrak{g}^{\lambda}\rightarrow\mathbb{C}$. Hence the associated $2$-tensor $\sum_{\alpha\in R_\lambda}\textup{Ad}_g(e_{-\alpha})\otimes\textup{Ad}_g(e_\alpha)$
does not depend on $g\in M$.
 \end{proof}
The adjoint action of $N_K(A)$ on 
$(\mathfrak{g}\otimes\mathfrak{g})^M$ and $U(\mathfrak{g})^M$ factorises to a $W$-action.
We write $w\cdot x$ for the action of $w\in W$ on $x\in (\mathfrak{g}\otimes\mathfrak{g})^M$
or $x\in U(\mathfrak{g})^M$. 
\begin{lem}\label{Nfactors}
We have
\begin{enumerate}
\item $w\cdot\varpi_\lambda=\varpi_{w\lambda}$ and $w\cdot\Omega_\lambda=
\Omega_{w\lambda}$ for $w\in W$ and $\lambda\in\Sigma$.
\item $\varpi_{\mathfrak{m}}\in S^2(\mathfrak{m})^{N_K(A)}$ and $\Omega_{\mathfrak{m}}\in
U(\mathfrak{m})^{N_K(A)}$.
\item $\sum_{i=1}^rx_i\otimes x_i\in S^2(\mathfrak{a})^{N_K(A)}$
and $\sum_{i=1}^rx_i^2\in U(\mathfrak{a})^{N_K(A)}$.
\end{enumerate}
\end{lem}
\begin{proof}
(1) For $w=gM\in W$ ($g\in N_K(A)$) the $\textup{Ad}_g$-invariance of $K_{\mathfrak{g}}$ implies
that the basis $\{\textup{Ad}_g(e_{-\alpha})\}_{\alpha\in R_\lambda}$ of $\mathfrak{g}^{-w\lambda}$ is dual to the basis
$\{\textup{Ad}_g(e_{\alpha})\}_{\alpha\in R_\lambda}$ of $\mathfrak{g}^{w\lambda}$ with respect to the perfect pairing
$K_{\mathfrak{g}}(\cdot,\cdot)|_{\mathfrak{g}^{-w\lambda}\times\mathfrak{g}^{w\lambda}}$. The result now follows immediately, cf. the proof of Lemma \ref{Mfactors}.\\
(2) This follows from the $\textup{Ad}(N_K(A))$-invariance of $\mathfrak{m}$ and
$K_{\mathfrak{g}}(\cdot,\cdot)|_{\mathfrak{m}\times\mathfrak{m}}$.\\
(3) The proof is similar to the proof of (2), now considering the $\textup{Ad}(N_K(A))$-invariant nondegenerate symmetric bilinear form
$K_{\mathfrak{g}}(\cdot,\cdot)|_{\mathfrak{a}\times\mathfrak{a}}$ on $\mathfrak{a}$.
\end{proof}

\begin{rema}\label{ONBexpression}
For $\lambda\in\Sigma^+$ let $\{b_{\lambda,i}\}_{i=1}^{\textup{mtp}(\lambda)}$ be an orthonormal basis of $\mathfrak{g}_{\mathbb{R}}^\lambda$ with respect to 
the restriction of the scalar product $(\cdot,\cdot)$ to $\mathfrak{g}_{\mathbb{R}}^\lambda$. Define $b_{-\lambda,i}:=-\theta_{\mathbb{R}}(b_{\lambda,i})$ for $i=1,\ldots,\textup{mtp}(\lambda)$,
then $\{b_{-\lambda,i}\}_{i=1}^{\textup{mtp}(\lambda)}$ is an orthonormal basis of $\mathfrak{g}_{\mathbb{R}}^{-\lambda}$ with respect to the restriction of $(\cdot,\cdot)$ to $\mathfrak{g}_{\mathbb{R}}^{-\lambda}$. In particular,
$\{b_{-\lambda,i}\}_{i=1}^{\textup{mtp}(\lambda)}$
and $\{b_{\lambda,i}\}_{i=1}^{\textup{mtp}(\lambda)}$ are bases of $\mathfrak{g}^{-\lambda}$ and $\mathfrak{g}^\lambda$ respectively, dual with respect to 
$K_{\mathfrak{g}}(\cdot,\cdot)|_{\mathfrak{g}^{-\lambda}\times\mathfrak{g}^{\lambda}}$. 
We conclude that
\begin{equation}\label{HAexpression}
\varpi_\lambda=\sum_{i=1}^{\textup{mtp}(\lambda)}b_{-\lambda,i}\otimes b_{\lambda,i}
\end{equation}
for $\lambda\in\Sigma$. This expression of $\varpi_\lambda$ is commonly used in the radial component computation of the Casimir element,
see, e.g., \cite[\S 9.1.2]{W}. Note that by \eqref{HAexpression},
\[
(\theta\otimes\textup{id})(\varpi_\lambda)=-\sum_{i=1}^{\textup{mtp}(\lambda)}b_{\lambda,i}\otimes
b_{\lambda,i}\in S^2(\mathfrak{g}^\lambda)
\]
and $(\theta\otimes\textup{id})(\varpi_\lambda)=(\textup{id}\otimes\theta)(\varpi_{-\lambda})$.
\end{rema}
For $\nu\in\mathfrak{a}^*$ we denote by $t_\nu$ the unique element in $\mathfrak{a}$
such that $K_{\mathfrak{g}}(t_\nu,h)=\nu(h)$ for all $h\in\mathfrak{a}$. Note that $w\cdot t_\nu=t_{w\nu}$ for $w\in W$, where 
the $W$-action on $\mathfrak{a}_{\mathbb{R}}$ and $\mathfrak{a}_{\mathbb{R}}^*$ is extended complex linearly to $\mathfrak{a}$ and $\mathfrak{a}^*$, respectively.

\begin{lem}\label{elh}\label{Hcases}
Let $\lambda\in\Sigma$.
\begin{enumerate}
\item If $\alpha\in R_0$ then $h_\alpha\in\mathfrak{t}$.
\item If $\alpha\in R_\lambda$ then $\textup{pr}_{\mathfrak{p}}(h_\alpha)=t_\lambda$.
\item 
$\sum_{\alpha\in R_\lambda}h_\alpha\in \mathfrak{g}^M$. 
\end{enumerate}
\end{lem}
\begin{proof}
(1) This follows from the observation that $\mathfrak{t}=\{h\in\mathfrak{h} \,\,\, | \,\,\, K_{\mathfrak{g}}(h,\mathfrak{a})=0\}$.\\
(2) Let $\alpha\in R_\lambda$. Then $\textup{pr}_{\mathfrak{p}}(h_\alpha)\in\mathfrak{a}$ satisfies for all $h\in\mathfrak{a}$,
\[
K_{\mathfrak{g}}(\textup{pr}_{\mathfrak{p}}(h_\alpha),h)=K_{\mathfrak{g}}(h_\alpha,h)=
\alpha(h)=\lambda(h),
\]
hence $\textup{pr}_{\mathfrak{p}}(h_\alpha)=t_\lambda$.\\
(3) Note that the $G$-equivariant linear map $\mathfrak{g}\otimes\mathfrak{g}\rightarrow \mathfrak{g}$, $x\otimes y\mapsto [x,y]$ 
maps $\varpi_\lambda$ to $-\sum_{\alpha\in R_\lambda}h_\alpha$. The result now follows from
Lemma \ref{Mfactors}.
\end{proof}
\begin{cor}\label{lempropertiescor}
We have 
\begin{equation}\label{division}
\sum_{\alpha\in R_\lambda}h_\alpha=
\textup{mtp}(\lambda)t_\lambda.
\end{equation}
\end{cor}
\begin{proof}
Recall that $\mathfrak{h}=\mathfrak{t}\oplus\mathfrak{a}$ is the decomposition of the $\theta$-stable Cartan subalgebra
$\mathfrak{h}$ in $(\pm 1)$-eigenspaces of $\theta\vert_{\mathfrak{h}}$. The resulting projections $\mathfrak{h}\twoheadrightarrow\mathfrak{t}$
and $\mathfrak{h}\twoheadrightarrow\mathfrak{a}$ are $\textup{pr}_{\mathfrak{k}}\vert_{\mathfrak{h}}$ and $\textup{pr}_{\mathfrak{p}}\vert_{\mathfrak{h}}$,
respectively. In view of Lemma \ref{elh}(2) it then suffices to show that
\[
\sum_{\alpha\in R_\lambda}h_\alpha\in\mathfrak{p},
\]
for which we resort to Remark \ref{ONBexpression}.
By \eqref{HAexpression} we have
\[
-\varpi_\lambda=\sum_{i=1}^{\textup{mtp}(\lambda)}\theta(b_{\lambda,i})\otimes b_{\lambda,i}
\]
and hence, by the proof of Lemma \ref{elh}(3),
\[
\sum_{\alpha\in R_\lambda}h_\alpha=\sum_{i=1}^{\textup{mtp}(\lambda)}[\theta(b_{\lambda,i}),b_{\lambda,i}].
\]
The right hand side of this expression manifestly lies in $\mathfrak{p}$.
\end{proof}

Consider the tensor product $W$-module algebra $\mathcal{R}\otimes U(\mathfrak{g})^M$. Its subalgebra $(\mathcal{R}\otimes U(\mathfrak{g})^M)^W$ of $W$-invariant elements contains  $U(\mathfrak{g})^{N_K(A)}$ as a subalgebra. The Casimir $\Omega$ decomposes as $\Omega=\Omega_{\mathfrak{m}}+\Omega^\prime$ with $\Omega_{\mathfrak{m}}, \Omega^\prime\in U(\mathfrak{g})^{N_K(A)}$. The second term $\Omega^\prime\in U(\mathfrak{g})^{N_K(A)}$ admits the following dynamical factorisation  
within $\bigl(\mathcal{R}\otimes U(\mathfrak{g})^M\bigr)^W$.
\begin{prop}\label{dynCasimirlemma}
We have 
\begin{equation}\label{dynC}
\Omega^\prime=
\sum_{j=1}^rx_j^2+\frac{1}{2}\sum_{\lambda\in\Sigma}\Bigl(
\frac{1+\xi_{2\lambda}}{1-\xi_{2\lambda}}\Bigr)\textup{mtp}(\lambda)t_\lambda
+2\sum_{\lambda\in\Sigma}\frac{\Omega_\lambda}{1-\xi_{2\lambda}}
\end{equation}
in $(\mathcal{R}\otimes U(\mathfrak{g})^M)^W$.
\end{prop}
\begin{proof}
By Lemma \ref{Nfactors} 
it is clear that the right-hand side of \eqref{dynC} lies in $(\mathcal{R}\otimes U(\mathfrak{g})^M)^W$.
By \eqref{division} we have
\begin{equation}\label{hdyn}
\Omega_{-\lambda}=\Omega_\lambda
+\textup{mtp}(\lambda)t_\lambda
\end{equation}
for $\lambda\in\Sigma$. Hence
\begin{equation*}
\begin{split}
\sum_{\lambda\in\Sigma}\frac{\Omega_\lambda}{1-\xi_{2\lambda}}=&
\sum_{\lambda\in\Sigma^+}\frac{
\textup{mtp}(\lambda)t_\lambda}{1-\xi_{-2\lambda}}+\sum_{\lambda\in\Sigma^+}\Bigl(\frac{1}{1-\xi_{2\lambda}}+\frac{1}{1-\xi_{-2\lambda}}\Bigr)
\Omega_\lambda\\
=&\sum_{\lambda\in\Sigma^+}\frac{
textup{mtp}(\lambda)t_\lambda}{1-\xi_{-2\lambda}}+\sum_{\lambda\in\Sigma^+}\Omega_\lambda\\
=&\sum_{\lambda\in\Sigma^+}\Bigl(\frac{1}{1-\xi_{-2\lambda}}-\frac{1}{2}\Bigr)
\textup{mtp}(\lambda)t_\lambda+\frac{1}{2}\sum_{\lambda\in\Sigma}\Omega_\lambda\\
=&-\frac{1}{4}\sum_{\lambda\in\Sigma}\Bigl(
\frac{1+\xi_{2\lambda}}{1-\xi_{2\lambda}}\Bigr)
\textup{mtp}(\lambda)t_\lambda
+\frac{1}{2}\sum_{\lambda\in\Sigma}\Omega_\lambda,
\end{split}
\end{equation*}
where we have used \eqref{hdyn} in the first and third equality. Substituting this identity in the 
right-hand side of the dynamical expression \eqref{dynC} reduces it to the expression \eqref{Omegam} for $\Omega^\prime$.
\end{proof}
\begin{rema} 
There is no dynamical expression for $\varpi^\prime$ in 
$(\mathcal{R}\otimes (\mathfrak{g}\otimes\mathfrak{g})^M)^W$ that leads to \eqref{dynC} 
by applying the multiplication map $\mu$ of $U(\mathfrak{g})$. We can however write
\[
\Omega^\prime=
-\Omega_{\mathfrak{m}}
+\frac{1}{2}\sum_{\lambda\in\Sigma}\Bigl(\frac{1+\xi_{2\lambda}}{1-\xi_{2\lambda}}\Bigr)\textup{mtp}(\lambda)t_\lambda-2\mu(r_{21})
\]
involving the dynamical element
\begin{equation}\label{rmatrix}
r=-\frac{1}{2}\varpi_{\mathfrak{m}}-\frac{1}{2}\sum_{j=1}^rx_j\otimes x_j
-\sum_{\lambda\in\Sigma}\frac{\varpi_\lambda}{1-\xi_{-2\lambda}}\in (\mathcal{R}\otimes (\mathfrak{g}\otimes\mathfrak{g})^M)^W
\end{equation}
(here $r_{21}$  is obtained from $r$ by interchanging its tensor components). The $2$-tensor $r$ will play an important role in the explicit description of the asymptotic boundary KZB operators,
see Theorem \ref{mainbKZB} and Proposition \ref{folding}. Note that $r$ satisfies the quasi-unitarity condition
\[
r+r_{21}=-\varpi\in S^2(\mathfrak{g})^G.
\]
In the real split case $r$ is Felder's \cite{F} trigonometric solution of the classical dynamical Yang-Baxter equation. 
We will show in Subsection \ref{Further} that for arbitrary $G$, folded versions of $r$ satisfy coupled classical dynamical Yang-Baxter-reflection equations.
\end{rema}

\subsection{The radial component of the action of the Casimir element}
We first define an explicit
second order differential operator $H^{\sigma_\ell,\underline{\tau},\sigma_r}\in\mathbb{D}_M((V_\ell\otimes\underline{U}\otimes V_r^*)^{\mathfrak{m}})^W$ and then show that the differential operator $D_{\Omega,n}^{\sigma_\ell,\underline{\tau},\sigma_r}$ is equal to $H^{\sigma_{\ell},\underline{\tau},\sigma_r}$. The technique to compute $D_{\Omega,n}^{\sigma_\ell,\underline{\tau},\sigma_r}$ goes back to Harish-Chandra, see, e.g., \cite[\S 9.1.2]{W}. 

Write for $\lambda\in\Sigma$,
\[
\upsilon_\lambda:=4(\textup{pr}_{\mathfrak{k}}\otimes\textup{pr}_{\mathfrak{k}})(\varpi_\lambda)=
\sum_{\alpha\in R_\lambda}y_{-\alpha}\otimes y_{\alpha}\in \mathfrak{q}\otimes\mathfrak{q}
\]
and set $\Upsilon_\lambda:=\mu(\upsilon_\lambda)=\sum_{\alpha\in R_\lambda}y_{-\alpha}y_{\alpha}\in U(\mathfrak{k})$. 
\begin{lem}\label{elemproperties}
Let $\lambda\in\Sigma$.
\begin{enumerate}
\item $\upsilon_\lambda=\upsilon_{-\lambda}$ and $\Upsilon_\lambda=\Upsilon_{-\lambda}$.
\item $\upsilon_\lambda\in (\mathfrak{q}\otimes\mathfrak{q})^M$ and
$\Upsilon_\lambda\in U(\mathfrak{k})^M$.
\item $w\cdot \upsilon_\lambda=\upsilon_{w\lambda}$ and $w\cdot\Upsilon_\lambda=
\Upsilon_{w\lambda}$ for $w\in W$.
\end{enumerate}
\end{lem}
\begin{proof}
(1) Observe that $(\textup{pr}_{\mathfrak{k}}\otimes\textup{pr}_{\mathfrak{k}})(\varpi_\lambda)=
(\textup{pr}_{\mathfrak{k}}\otimes\textup{pr}_{\mathfrak{k}})((\theta\otimes\textup{id})(\varpi_\lambda))$
is a symmetric tensor by Remark \ref{ONBexpression}. Hence $\upsilon_{\lambda}=\upsilon_{-\lambda}$, and consequently $\Upsilon_\lambda=\Upsilon_{-\lambda}$.\\
(2) The $\textup{Ad}(M)$-invariance of $\upsilon_\lambda$ follows from Lemma \ref{Mfactors} and the fact that $\textup{pr}_{\mathfrak{k}}$ is $\textup{Ad}(K)$-equivariant. The $\textup{Ad}(M)$-invariance of $\Upsilon_\lambda\in U(\mathfrak{k})$ now follows immediately.\\
(3) This follows from Lemma \ref{Nfactors}(1).
\end{proof}
Note that $S(\Upsilon_\lambda)=\Upsilon_\lambda$ in view of 
Lemma \ref{elemproperties}(1).
\begin{defi}
Define $H^{\sigma_\ell,\underline{\tau},\sigma_r}\in\mathbb{D}_M((V_\ell\otimes\underline{U}\otimes V_r^*)^{\mathfrak{m}})^W$ by 
\begin{equation*}
\begin{split}
H^{\sigma_\ell,\underline{\tau},\sigma_r}:=&
\sum_{j=1}^r\partial_{x_j}^2+\frac{1}{2}\sum_{\lambda\in\Sigma}\textup{mtp}(\lambda)\Bigl(\frac{\xi_\lambda+
\xi_{-\lambda}}{\xi_\lambda-\xi_{-\lambda}}\Bigr)\partial_{t_\lambda}+\sigma_r^*(\Omega_{\mathfrak{m}})\\
&\quad-
\sum_{\lambda\in\Sigma}\frac{\sigma_{\ell;n}(\Upsilon_\lambda)+
\sigma_r^*(\Upsilon_\lambda)+
(\xi_\lambda+\xi_{-\lambda})(\sigma_{\ell;n}\otimes\sigma_r^*)(\upsilon_\lambda)}{(\xi_\lambda-\xi_{-\lambda})^2}.
\end{split}
\end{equation*}
\end{defi}
The linear operators $\sigma_r^*(\Omega_{\mathfrak{m}})$,
$\sigma_r^*(z_\lambda)$, $\sigma_{\ell;n}(\Upsilon_\lambda)$,
$\sigma_r^*(\Upsilon_\lambda)$ and $(\sigma_{\ell;n}\otimes\sigma_r^*)(\upsilon_\lambda)$ occurring in the definition of $H^{\sigma_\ell,\underline{\tau},\sigma_r}$
are viewed as linear operators on $(V_\ell\otimes\underline{U}\otimes V_r^*)^{\mathfrak{m}}$.
Note that this is well-defined and that $H^{\sigma_\ell,\underline{\tau},\sigma_r}$ lies in
$\mathbb{D}_M((V_\ell\otimes\underline{U}\otimes V_r^*)^{\mathfrak{m}})^W$  in view of Lemma \ref{elemproperties} and
Lemma \ref{Nfactors}.
We now have the following result,
cf. \cite[Prop. 9.1.2.11]{W}.
\begin{thm}\label{Hexplicit}
We have
\begin{equation}\label{eee}
D_{\Omega,n}^{\sigma_\ell,\underline{\tau},\sigma_r}=H^{\sigma_\ell,\underline{\tau},\sigma_r}_{M_\circ}
\end{equation}
in $\mathbb{D}_M((V_\ell\otimes\underline{U}\otimes V_r^*)^{\mathfrak{m}})^W$. 
\end{thm}
\begin{proof}
By Corollary \ref{Mcor} the differential operator
$D_{\Omega,n}^{\sigma_\ell,\underline{\tau},\sigma_r}$ lies in
$\mathbb{D}_M((V_\ell\otimes\underline{U}\otimes V_r^*)^{\mathfrak{m}})^W$.
Furthermore,
\[
D_{\Omega,n}^{\sigma_\ell,\underline{\tau},\sigma_r}=L_{\Pi(\Omega)}^{\sigma_{\ell;n},\sigma_r}=
L_{\Pi(\Omega^\prime)}^{\sigma_{\ell;n},\sigma_r}+\sigma_r^*(\Omega_{\mathfrak{m}}).
\]
Hence it suffices to show that $L_{\Pi(\Omega^\prime)}^{\sigma_{\ell;n},\sigma_r}$ is
the restriction to $C^{\infty}(A_{\textup{reg}};(V_\ell\otimes\underline{U}\otimes V_r^*)^{\mathfrak{m}})$
of the differential operator
\begin{equation*}
\sum_{j=1}^r\partial_{x_j}^2+\frac{1}{2}\sum_{\lambda\in\Sigma}\textup{mtp}(\lambda)\Bigl(\frac{\xi_\lambda+
\xi_{-\lambda}}{\xi_\lambda-\xi_{-\lambda}}\Bigr)\partial_{t_\lambda}
-\sum_{\lambda\in\Sigma}\frac{\sigma_{\ell;n}(\Upsilon_\lambda)+
\sigma_r^*(\Upsilon_\lambda)+
(\xi_\lambda+\xi_{-\lambda})(\sigma_{\ell;n}\otimes\sigma_r^*)(\upsilon_\lambda)}{(\xi_\lambda-\xi_{-\lambda})^2}.
\end{equation*}
For this it suffices to prove that
\begin{equation*}
\begin{split}
\Pi(\Omega^\prime)&=\sum_{j=1}^r1\otimes x_j^2\otimes (1\otimes_{U(\mathfrak{m})}1)+
\frac{1}{2}\sum_{\lambda\in\Sigma}\textup{mtp}(\lambda)\Bigl(\frac{\xi_\lambda+\xi_{-\lambda}}
{\xi_\lambda-\xi_{-\lambda}}\Bigr)\otimes t_\lambda\otimes (1\otimes_{U(\mathfrak{m})}1)\\
&
+\sum_{\lambda\in\Sigma}\frac{(\xi_\lambda+\xi_{-\lambda})}{(\xi_\lambda-\xi_{-\lambda})^2}
\otimes 1\otimes\upsilon_\lambda
-\sum_{\lambda\in\Sigma}\frac{1}{(\xi_\lambda-\xi_{-\lambda})^2}\otimes 1
\otimes \Bigl(\Upsilon_\lambda\otimes_{U(\mathfrak{m})}1+1\otimes_{U(\mathfrak{m})}
\Upsilon_\lambda\Bigr),
\end{split}
\end{equation*}
where, in the fourth term, $\upsilon_\lambda\in\mathfrak{q}\otimes\mathfrak{q}$ is viewed as 
element in $U(\mathfrak{k})\otimes_{U(\mathfrak{m})} U(\mathfrak{k})$ via the canonical injection
$\mathfrak{q}\otimes\mathfrak{q}\hookrightarrow U(\mathfrak{k})\otimes_{U(\mathfrak{m})} U(\mathfrak{k})$, $x\otimes y\mapsto x\otimes_{U(\mathfrak{m})}y$. 

This in turn will follow from the
identity
\begin{equation*}
\begin{split}
\Omega^\prime&=\sum_{j=1}^rx_j^2+\frac{1}{2}\sum_{\lambda\in\Sigma}\Bigl(\frac{a^\lambda+a^{-\lambda}}{a^\lambda-a^{-\lambda}}\Bigr)\textup{mtp}(\lambda)t_\lambda\\
&\quad+\sum_{\lambda\in\Sigma}\frac{(a^\lambda+a^{-\lambda})\mu((\textup{Ad}_{a^{-1}}\otimes\textup{Id})
(\upsilon_\lambda))-\textup{Ad}_{a^{-1}}(\Upsilon_\lambda)-
\Upsilon_\lambda}{(a^\lambda-a^{-\lambda})^2}
\end{split}
\end{equation*}
in $U(\mathfrak{g})$ for all $a\in A_{\textup{reg}}$. 

Fix $a\in A_{\textup{reg}}$ and $\lambda\in\Sigma$. By applying \eqref{edecom} to both $e_{-\alpha}$ and $e_\alpha$ and using Lemma \ref{elemproperties}(1) we get
\begin{equation*}
\begin{split}
\Omega_\lambda=\sum_{\alpha\in R_\lambda}e_{-\alpha}e_\alpha&=
\frac{(a^\lambda+a^{-\lambda})\mu(\textup{Ad}_{a^{-1}}\otimes\textup{Id})(\upsilon_\lambda)-
\textup{Ad}_{a^{-1}}(\Upsilon_\lambda)-\Upsilon_\lambda}{(a^\lambda-a^{-\lambda})^2}\\
&+
\frac{a^{-\lambda}}{(a^\lambda-a^{-\lambda})^2}\sum_{\alpha\in R_\lambda}[y_{-\alpha},\textup{Ad}_{a^{-1}}(y_{\alpha})].
\end{split}
\end{equation*}
In view of \eqref{Omegam} it thus suffices to show that 
 \begin{equation}\label{todoH}
 \sum_{\lambda\in\Sigma}\frac{a^{-\lambda}}{(a^\lambda-a^{-\lambda})^2}
 \sum_{\alpha\in R_\lambda}[y_{-\alpha},\textup{Ad}_{a^{-1}}(y_{\alpha})]=\frac{1}{2}\sum_{\lambda\in\Sigma}\Bigl(\frac{a^\lambda+a^{-\lambda}}{a^\lambda-a^{-\lambda}}\Bigr)\textup{mtp}(\lambda)t_\lambda.
 \end{equation}
Note that for $\alpha\in R_\lambda$,
\begin{equation}\label{tttt}
\begin{split}
[y_{-\alpha},&\textup{Ad}_{a^{-1}}(y_{\alpha})]=
-a^{-\lambda}h_\alpha-a^{\lambda}\theta(h_\alpha)+a^{-\lambda}[\theta(e_{-\alpha}),e_{\alpha}]
+a^{\lambda}[e_{-\alpha},\theta(e_{\alpha})]\\
&=-(a^\lambda+a^{-\lambda})\textup{pr}_{\mathfrak{k}}(h_\alpha)+(a^\lambda-a^{-\lambda})t_\lambda+
a^{-\lambda}[\theta(e_{-\alpha}),e_{\alpha}]
+a^{\lambda}[e_{-\alpha},\theta(e_{\alpha})].
\end{split}
\end{equation}
When substituting into the left-hand side of \eqref{todoH}, the first term in the second line of \eqref{tttt} does not contribute since
$\textup{pr}_{\mathfrak{k}}(\sum_{\alpha\in R_\lambda}h_\alpha)=0$ by (the proof of) Corollary \ref{lempropertiescor}. 
The second term in the second line of \eqref{tttt} contributes
\begin{equation*}
\begin{split}
\sum_{\lambda\in\Sigma}\textup{mtp}(\lambda)\frac{a^{-\lambda}(a^\lambda-a^{-\lambda})}{(a^\lambda-a^{-\lambda})^2}
t_\lambda&=\sum_{\lambda\in\Sigma^+}\textup{mtp}(\lambda)
\Bigl(\frac{1}{a^{2\lambda}-1}-\frac{1}{a^{-2\lambda}-1}\Bigr)t_\lambda\\
&=\sum_{\lambda\in\Sigma^+}\textup{mtp}(\lambda)\Bigl(\frac{a^{\lambda}+a^{-\lambda}}{a^{\lambda}-a^{-\lambda}}\Bigr)t_\lambda=
\frac{1}{2}\sum_{\lambda\in\Sigma}\textup{mtp}(\lambda)
\Bigl(\frac{a^{\lambda}+a^{-\lambda}}{a^{\lambda}-a^{-\lambda}}\Bigr)t_\lambda
\end{split}
\end{equation*}
since $t_{-\lambda}=-t_\lambda$ and
$\textup{mtp}(-\lambda)=\textup{mtp}(\lambda)$. It thus suffices to show that
\begin{equation}\label{todoH2}
 \sum_{\lambda\in\Sigma}\frac{a^{-\lambda}}{(a^\lambda-a^{-\lambda})^2}
 \sum_{\alpha\in R_\lambda}\bigl(a^{-\lambda}[\theta(e_{-\alpha}),e_{\alpha}]+a^{\lambda}[e_{-\alpha},\theta(e_{\alpha})]\bigr)
 =0,
\end{equation}
which in turn will follow from $\sum_{\alpha\in R_\lambda}[\theta(e_{-\alpha}),e_\alpha]=0$.
This though follows from the anti-symmetry of the Lie bracket and the fact that
\[
\sum_{\alpha\in R_\lambda}e_\alpha\otimes\theta(e_{-\alpha})\in S^2(\mathfrak{g}^\lambda),
\]
which was observed in Remark \ref{ONBexpression}.
\end{proof}
\begin{rema} 
In the real split case one can choose the root vectors $e_\alpha$ in such a way that $\theta(e_\alpha)=-e_{-\alpha}$ for $\alpha\in R$. In that case one has $\upsilon_\alpha=-y_\alpha\otimes y_\alpha$ and $\Upsilon_\alpha=-y_\alpha^2$ for $\alpha\in R=\Sigma$, and Theorem \ref{Hexplicit} reduces to \cite[Cor. 3.5]{SR}.
\end{rema}
\subsection{The Schr{\"o}dinger operator}
The following explicit gauge will turn $H^{\sigma_\ell,\underline{\tau},\sigma_r}$
into a Schr{\"o}dinger operator.

\begin{defi}\label{gaugeDef}
We write $D\mapsto\widetilde{D}$ for the $W$-equivariant $\mathcal{R}\otimes\textup{End}_M((V_\ell\otimes\underline{U}\otimes V_r^*)^{\mathfrak{m}})$-linear algebra automorphism of
$\mathbb{D}_M((V_\ell\otimes\underline{U}\otimes V_r^*)^{\mathfrak{m}})$ satisfying
\[
\widetilde{\partial_h}=\frac{1}{4}\sum_{\lambda\in\Sigma}\lambda(h)\textup{mtp}(\lambda)
\Bigl(\frac{1+\xi_{2\lambda}}{1-\xi_{2\lambda}}\Bigr)+\partial_h\qquad \forall\, h\in\mathfrak{a}_{\mathbb{R}}.
\]
\end{defi}
\begin{rema}
Consider the positive Weyl chamber
\[
A^+:=\{a\in A \,\,\, | \,\,\, a^\lambda>1\quad \forall\, \lambda\in\Sigma^+\}
\]
and define an analytic function $\delta$ on $A^+$ by
\begin{equation}\label{deltam}
\delta(a):=a^{\rho}\prod_{\lambda\in\Sigma^+}(1-a^{-2\lambda})^{\frac{\textup{mtp}(\lambda)}{2}},
\qquad \rho:=\frac{1}{2}\sum_{\lambda\in\Sigma^+}\textup{mtp}(\lambda)\lambda.
\end{equation}
A straightforward computation then shows that 
\[
\bigl(\widetilde{D}(\delta f)\bigr)(a)=
\delta(a)\bigl(Df\bigr)(a)
\] 
for $D\in\mathbb{D}_M((V_\ell\otimes\underline{U}\otimes V_r^*)^{\mathfrak{m}})$, 
$f\in C^\infty(A^+;(V_\ell\otimes\underline{U}\otimes V_r^*)^{\mathfrak{m}})$ and $a\in A^+$.
\end{rema}

Define the potential $V^{\sigma_\ell,\underline{\tau},\sigma_r}\in\bigl(\mathcal{R}\otimes\textup{End}_M((V_\ell\otimes\underline{U}\otimes V_r^*)^{\mathfrak{m}})\bigr)^W$
by
\begin{equation}\label{potential}
V^{\sigma_\ell,\underline{\tau},\sigma_r}:=-\sum_{\lambda\in\Sigma}\frac{k(\lambda)}{(\xi_\lambda-\xi_{-\lambda})^2}
-\sum_{\lambda\in\Sigma}\frac{
\sigma_{\ell;n}(\Upsilon_\lambda)+
\sigma_r^*(\Upsilon_\lambda)+
(\xi_\lambda+\xi_{-\lambda})(\sigma_{\ell;n}\otimes\sigma_r^*)(\upsilon_\lambda)}{(\xi_\lambda-\xi_{-\lambda})^2}
\end{equation}
with $k(\lambda):=\textup{mtp}(\lambda)\bigl(\textup{mtp}(2\lambda)+\frac{\textup{mtp}(\lambda)}{2}-1\bigr)(t_\lambda,t_\lambda)$.
\begin{prop}\label{Sprop}
The gauged differential operator $\widetilde{H}^{\sigma_\ell,\underline{\tau},\sigma_r}\in\mathbb{D}_M((V_\ell\otimes\underline{U}\otimes V_r^*)^{\mathfrak{m}})^W$ is explicitly given by
\begin{equation}\label{tildeH}
\widetilde{H}^{\sigma_\ell,\underline{\tau},\sigma_r}=
\sum_{j=1}^r\partial_{x_j}^2+V^{\sigma_\ell,\underline{\tau},\sigma_r}+\sigma_r^*(\Omega_{\mathfrak{m}})-(t_\rho,t_\rho).
\end{equation}
\end{prop}
\begin{proof}
The proof of the first statement is similar to the scalar case ($n=0$, $\sigma_\ell=\sigma_r$ the trivial representation), cf. the proof of \cite[Thm. 2.1.1]{HS}. 
\end{proof}
Write $\widetilde{A}^{\sigma_\ell,\underline{\tau},\sigma_r}$ and $\widetilde{B}^{\sigma_\ell,\underline{\tau},\sigma_r}$ for the images of the algebras $A^{\sigma_\ell,\underline{\tau},\sigma_r}$ and
$B^{\sigma_\ell,\underline{\tau},\sigma_r}$ (see \eqref{incalgebra}) 
under the isomorphism $D\mapsto\widetilde{D}$ of 
$\mathbb{D}((V_\ell\otimes\underline{U}\otimes V_r^*)^M)^W$. We have an inclusion of algebras 
\begin{equation}\label{incalgebratilde}
\widetilde{A}^{\sigma_\ell,\underline{\tau},\sigma_r}\subseteq
\widetilde{B}^{\sigma_\ell,\underline{\tau},\sigma_r}\subset\mathbb{D}((V_\ell\otimes\underline{U}
\otimes V_r^*)^M)^W,
\end{equation}
with $\widetilde{A}^{\sigma_\ell,\underline{\tau},\sigma_r}$ commutative and containing
$\widetilde{H}_M^{\sigma_\ell,\underline{\tau},\sigma_r}=\widetilde{D}_{\Omega,n;M}^{\sigma_\ell,\underline{\tau},\sigma_r}$.
\begin{defi}\label{Sdef}
The quantum Calogero-Moser spin chain is the quantum superintegrable system with
algebra of quantum Hamiltonians $\widetilde{A}^{\sigma_\ell,\underline{\tau},\sigma_r}$,
algebra of quantum integrals $\widetilde{B}^{\sigma_\ell,\underline{\tau},\sigma_r}$ and Schr{\"o}dinger operator 
$\widetilde{H}_M^{\sigma_\ell,\underline{\tau},\sigma_r}
\in\widetilde{A}^{\sigma_\ell,\underline{\tau},\sigma_r}$.
\end{defi}
The quantum Calogero-Moser spin chain is a mixture of a quantum spin Calogero-Moser system
of rank $r$ and a one-dimensional spin chain with $n$ internal sites and two reflecting boundaries. For the bosonic theory, the quantum state space is a suitable completion of the space $C^\infty(A_{\textup{reg}};(V_\ell\otimes\underline{U}\otimes V_r^*)^M)^W$. 
\begin{cor}\label{npointcor}
Let $f\in C_{\sigma_\ell,\underline{\tau},\sigma_r}^\infty(G^{\times (n+1)};\pmb{\chi})$ be an
$n$-point spherical function \textup{(}see Definition \ref{defnpoint}\textup{)}. Write
\[
\widetilde{f}^\flat\in C^\infty(A_{\textup{reg}};(V_\ell\otimes\underline{U}\otimes V_r^*)^M)^W
\]
for the $W$-invariant $(V_\ell\otimes\underline{U}\otimes V_r^*)^M$-valued function such that
\[
\widetilde{f}^\flat(a):=\delta(a)f^\flat(a)\qquad \forall\, a\in A^+.
\]
For all $\pmb{X}\in Z(U(\mathfrak{g}))^{\otimes (n+1)}$ we have 
\[
\widetilde{D}_{\pmb{X};M}^{\sigma_\ell,\underline{\tau},\sigma_r}\bigl(\widetilde{f}^\flat\bigr)=
\pmb{\chi}(\pmb{X})\widetilde{f}^\flat,
\]
with $\pmb{\chi}$ viewed as the character $\chi_0\otimes\cdots\otimes\chi_n$ of
$Z(U(\mathfrak{g}))^{\otimes (n+1)}$. In particular,
\[
\widetilde{H}_M^{\sigma_\ell,\underline{\tau},\sigma_r}\bigl(\widetilde{f}^\flat\bigr)=
\chi_n(\Omega)\widetilde{f}^\flat.
\]
\end{cor}
See \cite{SR} for explicit examples of classes of $n$-point spherical functions when $G$ is real split.

\section{The asymptotic boundary KZB operators}\label{SectionabKZB}
Fix throughout this section finite-dimensional $K$-representations $(\sigma_\ell, V_\ell)$, 
$(\sigma_r,V_r)$ and finite-dimensional $G$-representations ($\tau_i,U_i)$ ($1\leq i\leq n$).

The quadratic Hamiltonians $\widetilde{D}_{\Omega,j;M}^{\sigma_\ell,\underline{\tau},\sigma_r}\in\widetilde{A}^{\sigma_\ell,\underline{\tau},\sigma_r}$ ($0\leq j\leq n$) of the quantum Calogero-Moser spin chain are related to the Schr{\"o}dinger
operator $\widetilde{H}_M^{\sigma_\ell,\underline{\tau},\sigma_r}$ by
\[
\widetilde{D}_{\Omega,j;M}^{\sigma_\ell,\underline{\tau},\sigma_r}=
\widetilde{H}_{M}^{\sigma_\ell,\underline{\tau},\sigma_r}-
2\sum_{i=j+1}^n\widetilde{\mathcal{D}}_{j;M}^{\sigma_\ell,\underline{\tau},\sigma_r}
\in\widetilde{A}^{\sigma_\ell,\underline{\tau},\sigma_r}
\qquad (0\leq j\leq n),
\]
with $\widetilde{\mathcal{D}}_{i;M}^{\sigma_\ell,\underline{\tau},\sigma_r}\in\widetilde{A}^{\sigma_\ell,\underline{\tau},\sigma_r}$ the quantum Hamiltonian obtained by gauging and $M$-restricting the differential operator
\[
\mathcal{D}_i^{\sigma_\ell,\underline{\tau},\sigma_r}:=
\frac{1}{2}\bigl(D_{\Omega,i}^{\sigma_\ell,\underline{\tau},\sigma_r}-
D_{\Omega,i-1}^{\sigma_\ell,\underline{\tau},\sigma_r}\bigr)
\in\mathbb{D}_M((V_\ell\otimes\underline{U}\otimes V_r^*)^{\mathfrak{m}})^W
\qquad (1\leq i\leq n).
\]
\begin{defi}
The quantum Hamiltonians $\widetilde{\mathcal{D}}_{i;M}^{\sigma_\ell,\underline{\tau},\sigma_r}
\in\widetilde{A}^{\sigma_\ell,\underline{\tau},\sigma_r}$
\textup{(}$1\leq i\leq n$\textup{)} are called the asymptotic boundary KZB operators.
\end{defi}
We now first derive explicit expressions for the asymptotic boundary KZB operators.

\subsection{The derivation of the explicit expression}\label{subsectionKZB}
The asymptotic boundary KZB operators turn out to be first-order differential operators, 
whose local terms we define now first. Write
\begin{equation*}
\begin{split}
r^+&:=-\varpi_{\mathfrak{m}}-\sum_{\lambda\in\Sigma}\frac{
2(\textup{pr}_{\mathfrak{k}}\otimes\textup{id})(\varpi_\lambda)}{1-\xi_{-2\lambda}}\in
(\mathcal{R}\otimes (\mathfrak{k}\otimes\mathfrak{g})^M)^W,\\
r^-&:=\sum_{j=1}^rx_j\otimes x_j+\sum_{\lambda\in\Sigma}
\frac{2(\textup{pr}_{\mathfrak{p}}\otimes\textup{id})(\varpi_\lambda)}{1-\xi_{-2\lambda}}\in (\mathcal{R}\otimes (\mathfrak{p}\otimes\mathfrak{g})^M)^W
\end{split}
\end{equation*}
(the $M$-invariance follows from Lemma \ref{Mfactors}, and the $W$-invariance from Lemma \ref{Nfactors}). Define $\kappa^{\textup{core}}\in (\mathcal{R}\otimes U(\mathfrak{g})^M)^W$ by
\begin{equation}\label{kappacore}
\kappa^{\textup{core}}:=\frac{1}{2}\Omega+\mu(r_{21}^+)=-\frac{1}{2}\Omega_{\mathfrak{m}}
+\frac{1}{2}\Omega^\prime-\sum_{\lambda\in\Sigma}\sum_{\alpha\in R_\lambda}
\frac{e_{-\alpha}y_\alpha}{1-\xi_{2\lambda}}
\end{equation}
and $\kappa\in (\mathcal{R}\otimes (U(\mathfrak{k})\otimes U(\mathfrak{g})\otimes U(\mathfrak{k}))^M)^W$ by
\[
\kappa:=1\otimes\kappa^{\textup{core}}\otimes 1
-\Bigl(\varpi_{\mathfrak{m}}+\sum_{\lambda\in\Sigma}\frac{2(\textup{pr}_{\mathfrak{k}}\otimes\textup{id})(\varpi_\lambda)}{1-\xi_{-2\lambda}}
\Bigr)\otimes 1
-1\otimes\Bigl(\sum_{\lambda\in\Sigma}\frac{2(\textup{id}\otimes\textup{pr}_{\mathfrak{k}})(\varpi_\lambda)}{\xi_{-\lambda}-\xi_\lambda}\Bigr).
\]
For $1\leq i<j\leq n$ and $1\leq s\leq n$ we write
\[
r_{ij}^{\pm}:=(\tau_i\otimes\tau_j)(r^{\pm}),\qquad 
\kappa_s:=(\sigma_\ell\otimes\tau_s\otimes\sigma_r^*)(\kappa),
\]
which we view as elements in 
$(\mathcal{R}\otimes\textup{End}_M((V_\ell\otimes\underline{U}\otimes V_r^*)^{\mathfrak{m}}))^W$.
\begin{defi}\label{Ddef}
Define first-order differential operators 
$\mathcal{D}_i^{\sigma_\ell,\underline{\tau},\sigma_r}\in \mathbb{D}_M((V_\ell\otimes\underline{U}\otimes V_r^*)^{\mathfrak{m}})^W$ \textup{(}$1\leq i\leq n$\textup{)} by
\[
\mathcal{D}_i^{\sigma_\ell,\underline{\tau},\sigma_r}=
\sum_{j=1}^r\tau_i(x_j)\partial_{x_j}-
\kappa_i-\sum_{k=1}^{i-1}r_{ki}^+-\sum_{k=i+1}^nr_{ik}^-.
\]
\end{defi}
The following theorem is the main result of this section.
\begin{thm}\label{mainbKZB}
For $1\leq i\leq n$ we have the identity
\begin{equation}\label{keyformula}
\frac{1}{2}\bigl(D_{\Omega,i}^{\sigma_\ell,\underline{\tau},\sigma_r}-D_{\Omega,i-1}^{\sigma_\ell,\underline{\tau},\sigma_r}\bigr)=\mathcal{D}_{i}^{\sigma_\ell,\underline{\tau},\sigma_r}
\end{equation}
in $\mathbb{D}_M\bigl((V_\ell\otimes\underline{U}\otimes V_r^*)^{\mathfrak{m}}\bigr)^W$. 
\end{thm}
Before discussing the proof of Theorem \ref{mainbKZB}, we note some immediate consequences first.
\begin{cor}\label{mainbKZBcor}
\textup{(1)} The $M$-restricted gauged differential operators
$\widetilde{\mathcal{D}}_{i;M}^{\sigma_\ell,\underline{\tau},\sigma_r}\in\widetilde{A}^{\sigma_\ell,\underline{\tau},\sigma_r}$ are first-order quantum Hamiltonians for the quantum Calogero-Moser spin chain. In particular, 
\begin{equation}\label{commutativity}
\lbrack \widetilde{\mathcal{D}}_{i;M}^{\sigma_\ell,\underline{\tau},\sigma_r},
\widetilde{\mathcal{D}}_{j;M}^{\sigma_\ell,\underline{\tau},\sigma_r}\rbrack=0,
\qquad \lbrack\widetilde{\mathcal{D}}_{i;M}^{\sigma_\ell,\underline{\tau},\sigma_r},
\widetilde{H}_M^{\sigma_\ell,\underline{\tau},\sigma_r}\rbrack=0
\end{equation}
for $1\leq i,j\leq n$.\\
\textup{(2)} The quadratic Hamiltonian $\widetilde{D}_{\Omega,j;M}^{\sigma_\ell,\underline{\tau},\sigma_r}\in\widetilde{A}^{\sigma_\ell,\underline{\tau},\sigma_r}$ for $j\in\{0,\ldots,n\}$ is explicitly given by
\[
\widetilde{D}_{\Omega,j;M}^{\sigma_\ell,\underline{\tau},\sigma_r}=
\widetilde{H}_{M}^{\sigma_\ell,\underline{\tau},\sigma_r}-
2\sum_{i=j+1}^n\widetilde{\mathcal{D}}_{j;M}^{\sigma_\ell,\underline{\tau},\sigma_r}.
\]
\textup{(3)} We have
\[
\widetilde{\mathcal{D}}_{i;M}^{\sigma_\ell,\underline{\tau},\sigma_r}\bigl(\widetilde{f}^\flat\bigr)=
\frac{1}{2}\bigl(\chi_i(\Omega)-\chi_{i-1}(\Omega)\bigr)\widetilde{f}^\flat
\]
for $f\in C_{\sigma_\ell,\underline{\tau},\sigma_r}^{\infty}(G^{\times (n+1)};\pmb{\chi})$
and $1\leq i\leq n$.\\
\textup{(4)} We have in $\mathbb{D}_M((V_\ell\otimes\underline{U}\otimes V_r^*)^{\mathfrak{m}})^W$,
\begin{equation}\label{gauge}
\widetilde{\mathcal{D}}_i^{\sigma_\ell,\underline{\tau},\sigma_r}=
\sum_{j=1}^r\tau_i(x_j)\partial_{x_j}-
\widehat{\kappa}_i-\sum_{k=1}^{i-1}r_{ki}^+-\sum_{k=i+1}^nr_{ik}^-
\end{equation}
with 
\[
\widehat{\kappa}_i=\kappa_i-\frac{1}{4}\sum_{\lambda\in\Sigma}\Bigl(\frac{1+\xi_{2\lambda}}{1-\xi_{2\lambda}}\Bigr)\textup{mtp}(\lambda)\tau_i(t_\lambda).
\]
\end{cor}

\begin{rema}
It follows from Corollary \ref{mainbKZBcor}(4) that $\widetilde{\mathcal{D}}_i^{\sigma_\ell,\underline{\tau},\sigma_r}\in \mathbb{D}_M((V_\ell\otimes\underline{U}\otimes V_r^*)^{\mathfrak{m}})^W$
is obtained from $\mathcal{D}_i^{\sigma_\ell,\underline{\tau},\sigma_r}$ by replacing
the core $\kappa^{\textup{core}}$ \eqref{kappacore} of $\kappa$ by 
\begin{equation}\label{kappacore2}
\widehat{\kappa}^{\textup{core}}:=\kappa^{\textup{core}}
-\frac{1}{4}\sum_{\lambda\in\Sigma}\Bigl(\frac{1+\xi_{2\lambda}}{1-\xi_{2\lambda}}\Bigr)\textup{mtp}(\lambda)t_\lambda.
\end{equation}
\end{rema}

\begin{rema}
If the conjecture in Remark \ref{SplitRemark2} is valid, then the differential operators
$\widetilde{H}^{\sigma_\ell,\underline{\tau},\sigma_r}$, $\widetilde{\mathcal{D}}_{i}^{\sigma_\ell,\underline{\tau},\sigma_r}$ 
($1\leq i\leq n$) already pairwise commute in $\mathbb{D}_M((V_\ell\otimes\underline{U}\otimes V_r^*)^{\mathfrak{m}})^W$.
This is indeed the case when $G$ is real split, see \cite[\S 6]{SR}.
\end{rema}
We now prove Theorem \ref{mainbKZB} in a couple of steps.
\begin{lem}\label{DeltaOmega}
For $k\in\mathbb{Z}_{\geq 0}$,
\[
\Delta^{(k)}(\Omega)=
\sum_{j=1}^{k+1}\Omega_j+
2\sum_{1\leq j<j^\prime\leq k+1}\sum_{s=1}^{\textup{dim}(\mathfrak{h})}(z_s)_j(z_s)_{j^\prime}
+2\sum_{1\leq j<j^\prime\leq k+1}\sum_{\alpha\in R}(e_{-\alpha})_j(e_\alpha)_{j^\prime}
\]
in $U(\mathfrak{g})^{\otimes (k+1)}$, with the subindices indicating in which tensor component
of $U(\mathfrak{g})^{\otimes (k+1)}$ the elements are placed.
\end{lem} 
\begin{proof}
This follows by induction to $k$ using the expression \eqref{Omega} for $\Omega$.
\end{proof}
\begin{cor}\label{corgen}
For $0\leq i\leq n$ we have
\begin{equation*}
\begin{split}
D_{\Omega,i}^{\sigma_\ell,\underline{\tau},\sigma_r}&=
L_{\widetilde{\Pi}(\Omega)}^{\sigma_{\ell;n},\sigma_r}-
2\sum_{j=i+1}^n\sum_{s=1}^{\textup{dim}(\mathfrak{h})}\tau_j(z_s)
L_{\widetilde{\Pi}(z_s)}^{\sigma_{\ell;n},\sigma_r}
-2\sum_{j=i+1}^n\sum_{\alpha\in R}\tau_j(e_{-\alpha})L_{\widetilde{\Pi}(e_\alpha)}^{\sigma_{\ell;n},\sigma_r}\\
&+\sum_{j=i+1}^n\tau_j(\Omega)+
2\sum_{i+1\leq j<j^\prime\leq n}(\tau_j\otimes\tau_{j^\prime})(\varpi),
\end{split}
\end{equation*}
viewed as differential operators in $\mathbb{D}\bigl((V_\ell\otimes\underline{U}\otimes V_r^*)^{\mathfrak{m}},V_\ell\otimes\underline{U}\otimes V_r^*\bigr)$. 
\end{cor}
\begin{proof}
For $0\leq i<n$ this follows directly from Definition \ref{Dcoor}, Lemma \ref{DeltaOmega}, the identity $S(\Omega)=\Omega$ and \eqref{varpi}. For $i=n$ it follows
from 
Definition \ref{Dcoor} and \eqref{lrsame}.
\end{proof}
Write for $\lambda\in\Sigma$,
\[
\varpi_\lambda^r:=2(\textup{id}\otimes\textup{pr}_{\mathfrak{k}})(\varpi_\lambda)=
\sum_{\alpha\in R_\lambda} e_{-\alpha}\otimes y_\alpha\in (\mathfrak{g}\otimes\mathfrak{k})^M.
\]
Furthermore, for $X=\sum_ta_t\otimes b_t\in\mathfrak{g}\otimes\mathfrak{k}$ we write
$(\tau_i\otimes\sigma_{\ell;n})(X):=\sum_t\tau_i(a_t)\sigma_{\ell;n}(b_t)$, viewed as endomorphism
of $V_\ell\otimes\underline{U}\otimes V_r^*$ (cf. Remark \ref{conv}). 
\begin{lem}\label{technicallem}
We have for $1\leq i\leq n$, 
\begin{equation}\label{tl1}
\begin{split}
\frac{1}{2}(D_{\Omega,i}^{\sigma_\ell,\underline{\tau},\sigma_r}-D_{\Omega,i-1}^{\sigma_\ell,
\underline{\tau},\sigma_r})
=&
\sum_{j=1}^r\tau_i(x_j)\partial_{x_j}-\frac{1}{2}\tau_i(\Omega)-
\sum_{k=i+1}^n(\tau_i\otimes\tau_k)(\varpi)\\
+&(\tau_i\otimes\sigma_{\ell;n})(\varpi_{\mathfrak{m}})+
\sum_{\lambda\in\Sigma}\Bigl(\frac{(\tau_i\otimes\sigma_{\ell;n})(\varpi_\lambda^r)+
\xi_\lambda (\tau_i\otimes\sigma_r^*)(\varpi_\lambda^r)}{1-\xi_{2\lambda}}\Bigr),
\end{split}
\end{equation}
viewed as differential operators
in $\mathbb{D}\bigl((V_\ell\otimes\underline{U}\otimes V_r^*)^{\mathfrak{m}},V_\ell\otimes\underline{U}\otimes V_r^*\bigr)$.
\end{lem}
\begin{proof}
All identities below are in $\mathbb{D}((V_\ell\otimes\underline{U}\otimes V_r^*)^{\mathfrak{m}},V_\ell\otimes\underline{U}\otimes V_r^*)$, i.e., they hold true when acting on
functions $f\in C^\infty(A_{\textup{reg}};(V_\ell\otimes\underline{U}\otimes V_r^*)^{\mathfrak{m}})$.
By Corollary \ref{corgen},
\begin{equation}\label{Dinitial}
\begin{split}
\frac{1}{2}(D_{\Omega,i}^{\sigma_\ell,\underline{\tau},\sigma_r}-D_{\Omega,i-1}^{\sigma_\ell,
\underline{\tau},\sigma_r})
&=
\sum_{s=1}^{\textup{dim}(\mathfrak{h})}\tau_i(z_s)L_{\widetilde{\Pi}(z_s)}^{\sigma_{\ell;n},\sigma_r}+\sum_{\alpha\in R}\tau_i(e_{-\alpha})L_{\widetilde{\Pi}(e_\alpha)}^{\sigma_{\ell;n},\sigma_r}\\
&-\frac{1}{2}\tau_i(\Omega)-\sum_{k=i+1}^n(\tau_i\otimes\tau_k)(\varpi).
\end{split}
\end{equation}
Use that $L_{\widetilde{\Pi}(x_j)}^{\sigma_{\ell;n},\sigma_r}=\partial_{x_j}$
for $1\leq j\leq r$, $L_{\widetilde{\Pi}(z_s)}^{\sigma_{\ell;n},\sigma_r}=
\sigma_{\ell;n}(z_s)$ for $r+1\leq s\leq\textup{dim}(\mathfrak{h})$, 
$L_{\widetilde{\Pi}(e_\alpha)}^{\sigma_{\ell;n},\sigma_r}=
\sigma_{\ell;n}(e_\alpha)$ for $\alpha\in R_0$, and
\[
L_{\widetilde{\Pi}(e_\alpha)}^{\sigma_{\ell;n},\sigma_r}=
\frac{\sigma_{\ell;n}(y_\alpha)+\xi_\lambda\sigma_r^*(y_\alpha)}{1-\xi_{2\lambda}}
\]
for $\alpha\in R_\lambda$ with $\lambda\in\Sigma$ (the last formula is a consequence of \eqref{Pie}).
Substitute these equations in \eqref{Dinitial} and group the terms which act 
through $\tau_i$, through $\tau_i$ and $\tau_k$, through $\tau_i$ and $\sigma_{\ell;n}$, and through
$\tau_i$ and $\sigma_r^*$. This directly leads to the desired result. 
 \end{proof}
\noindent
{\bf Proof of Theorem \ref{mainbKZB}:} We write the degree zero terms of the differential operator on the right-hand side of \eqref{tl1} in terms of local factors.
By the fact that 
$\varpi_{\mathfrak{m}}\in S^2(\mathfrak{m})$, $\Omega_{\mathfrak{m}}=\mu(\varpi_{\mathfrak{m}})$ and the definition of
$\sigma_{\ell;n}$, we have 
\begin{equation*}
(\tau_i\otimes\sigma_{\ell;n})(\varpi_{\mathfrak{m}})=
(\sigma_\ell\otimes\tau_i)(\varpi_{\mathfrak{m}})+
\sum_{k=1}^{i-1}(\tau_k\otimes\tau_i)(\varpi_{\mathfrak{m}})+\tau_i(\Omega_{\mathfrak{m}})
+\sum_{k=i+1}^{n}(\tau_i\otimes\tau_k)(\varpi_{\mathfrak{m}})
\end{equation*}
in $\textup{End}(V_\ell\otimes\underline{U}\otimes V_r^*)$.
Similarly, we have for $\alpha\in R\setminus R_0$,
\begin{equation*}
\tau_i(e_{-\alpha})\sigma_{\ell;n}(y_\alpha)=
\sigma_\ell(y_\alpha)\tau_i(e_{-\alpha})
+\sum_{k=1}^{i-1}\tau_k(y_\alpha)\tau_i(e_{-\alpha})
+\tau_i(e_{-\alpha}y_\alpha)
+\sum_{k=i+1}^n\tau_i(e_{-\alpha})\tau_k(y_\alpha).
\end{equation*}
Substituting these identities into \eqref{tl1} and comparing the resulting expression with the explicit formula for $\mathcal{D}_i^{\sigma_\ell,\underline{\tau},\sigma_r}$
from Definition \ref{Ddef}, formula \eqref{keyformula} will follow from the identities
\begin{equation*}
\begin{split}
r_{ki}^+&=-(\tau_k\otimes\tau_i)(\varpi_{\mathfrak{m}})-
\sum_{\lambda\in\Sigma}\frac{1}{1-\xi_{2\lambda}}\sum_{\alpha\in R_\lambda}\tau_k(y_\alpha)\tau_i(e_{-\alpha}),\qquad 1\leq k<i,\\
r_{ik}^-&=(\tau_i\otimes\tau_k)(\varpi^\prime)-\sum_{\lambda\in\Sigma}
\frac{1}{1-\xi_{2\lambda}}\sum_{\alpha\in R_\lambda}\tau_i(e_{-\alpha})\tau_k(y_\alpha),
\qquad i<k\leq n,\\
\tau_i(\kappa^{\textup{core}})&=-\frac{1}{2}\tau_i(\Omega_{\mathfrak{m}})+\frac{1}{2}\tau_i(\Omega^\prime)-\sum_{\lambda\in\Sigma}
\frac{1}{1-\xi_{2\lambda}}\sum_{\alpha\in R_\lambda}\tau_i(e_{-\alpha}y_\alpha)
\end{split}
\end{equation*}
as endomorphisms of $V_\ell\otimes\underline{U}\otimes V_r^*$.
Only the second equality requires proof; it follows from the fact that
\begin{equation*}
\begin{split}
\varpi^\prime-\sum_{\lambda\in\Sigma}\frac{1}{1-\xi_{2\lambda}}\sum_{\alpha\in R_\lambda}
e_{-\alpha}\otimes y_\alpha&=\sum_{j=1}^rx_j\otimes x_j+\sum_{\lambda\in\Sigma}\sum_{\alpha\in R_\lambda}
e_{-\alpha}\otimes\Bigl(\frac{(1-\xi_{2\lambda})e_\alpha-y_\alpha}{1-\xi_{2\lambda}}\Bigr)\\
&=\sum_{j=1}^rx_j\otimes x_j-\sum_{\lambda\in\Sigma}\sum_{\alpha\in R_\lambda}
e_{-\alpha}\otimes\Bigl(\frac{\xi_{2\lambda}e_\alpha+c_\alpha e_{\theta\alpha}}{1-\xi_{2\lambda}}
\Bigr)\\
&=\sum_{j=1}^rx_j\otimes x_j+\sum_{\lambda\in\Sigma}\frac{1}{1-\xi_{-2\lambda}}
\sum_{\alpha\in R_\lambda}(e_{-\alpha}-c_\alpha^{-1}e_{-\theta\alpha})\otimes e_\alpha\\
&=r^-,
\end{split}
\end{equation*}
where we have used Lemma \ref{clem} for the third and fourth equality. It is clear that both sides of \eqref{keyformula} lie in $\mathbb{D}_M((V_\ell\otimes\underline{U}\otimes V_r^*)^{\mathfrak{m}})^W$. This completes the proof of the theorem. \qed

\vspace{.3cm}
The local factors $r^{\pm}$ and $\kappa^{\textup{core}}$ occurring in $\mathcal{D}_i^{\sigma_\ell,\underline{\tau},\sigma_r}\in \mathbb{D}_M((V_\ell\otimes\underline{U}\otimes V_r^*)^{\mathfrak{m}})^W$
($1\leq i\leq n$) decompose as follows.
\begin{prop}\label{folding}
We have
\begin{equation*}
\begin{split}
r^{\pm}&=\pm r+(\textup{id}\otimes\theta)r_{21},\\
\kappa^{\textup{core}}&=\mu\bigl((1\otimes\theta)r_{21})
+\frac{1}{4}\sum_{\lambda\in\Sigma}\Bigl(
\frac{1+\xi_{2\lambda}}{1-\xi_{2\lambda}}\Bigr)\textup{mtp}(\lambda)t_\lambda
\end{split}
\end{equation*}
with $r\in (\mathcal{R}\otimes (\mathfrak{g}\otimes\mathfrak{g})^M)^W$ given by
\eqref{rmatrix}.
Furthermore, $(\textup{id}\otimes\theta)r_{21}=(\theta\otimes\textup{id})r$.
\end{prop}
\begin{proof} 
The alternative expressions for $r^{\pm}$ follow by a straightforward computation (the derivation is similar to the last computation in the proof of Theorem \ref{mainbKZB}).
Proposition \ref{dynCasimirlemma} implies that
\begin{equation}\label{kappacorealt}
\kappa^{\textup{core}}=
-\frac{1}{2}\Omega_{\mathfrak{m}}\\
+\frac{1}{2}\sum_{j=1}^rx_j^2+\frac{1}{4}\sum_{\lambda\in\Sigma}\Bigl(
\frac{1+\xi_{2\lambda}}{1-\xi_{2\lambda}}\Bigr)\textup{mtp}(\lambda)t_\lambda
-\sum_{\lambda\in\Sigma}\frac{\mu((\textup{id}\otimes\theta)(\varpi_\lambda))}{1-\xi_{2\lambda}},
\end{equation}
from which the alternative expression for $\kappa^{\textup{core}}$ easily follows.
The symmetry $(\textup{id}\otimes\theta)r_{21}=(\theta\otimes\textup{id})r$ follows from the fact
 that $(\theta\otimes\textup{id})(\varpi_\lambda)\in S^2(\mathfrak{g})$ for
$\lambda\in\Sigma$, cf. Remark \ref{ONBexpression}.
\end{proof}
Combined with \eqref{kappacore2} we conclude that
\begin{equation}\label{kappacore3}
\widehat{\kappa}^{\textup{core}}=\mu((1\otimes\theta)r_{21}).
\end{equation}
\begin{rema}
For real split $G$ an explanation of
the folded structure of $r^{\pm}$ and $\kappa^{\textup{core}}$, as described by Proposition \ref{folding}, is given in \cite[\S 6]{SR} using quantum field theoretic arguments and the theory of $n$-point spherical functions. 
\end{rema}

\subsection{Coupled classical dynamical Yang-Baxter-reflection equations}
As a consequence of the commutativity of the $M$-restricted asymptotic boundary KZB operators
(cf. Corollary \ref{mainbKZBcor}) we obtain coupled classical dynamical Yang-Baxter-reflection equations for its local factors $r^{\pm}$ and $\widehat{\kappa}$. To formulate these equations, 
we define the classical dynamical reflection term $\mathcal{A}\in\mathcal{R}\otimes(U(\mathfrak{k})\otimes U(\mathfrak{g})^{\otimes 2}\otimes U(\mathfrak{k}))^M$ by
\begin{equation}\label{A}
\mathcal{A}:=[\widehat{\kappa}_1+r^-_{12},\widehat{\kappa}_2+r^+_{12}]-\sum_{j=1}^r\bigl((x_j)_1\partial_{x_j}(\widehat{\kappa}_2+r^+_{12})-(x_j)_2
\partial_{x_j}(\widehat{\kappa}_1+r^-_{12})\bigr),
\end{equation}
with the sublabels numbering the two $U(\mathfrak{g})$ tensor legs  
within $U(\mathfrak{k})\otimes U(\mathfrak{g})^{\otimes 2}\otimes U(\mathfrak{k})$,
and we define mixed classical dynamical Yang-Baxter terms 
$\textup{MYB}(s)\in\mathcal{R}\otimes (U(\mathfrak{g})^{\otimes 3})^M$ ($1\leq s\leq 3$) by
\begin{equation}\label{mCYB}
\begin{split}
\textup{MYB}(1)&:=[r_{12}^+,r_{13}^+]+[r_{12}^+,r_{23}^+]-[r_{13}^+,r_{23}^-]
-\sum_{j=1}^r\bigl((x_j)_2\partial_{x_j}(r_{13}^+)-(x_j)_3\partial_{x_j}(r_{12}^+)\bigr),\\
\textup{MYB}(2)&:=[r_{12}^-,r_{13}^+]+[r_{12}^-,r_{23}^+]+[r_{13}^-,r_{23}^+]
-\sum_{j=1}^r\bigl((x_j)_1\partial_{x_j}(r_{23}^+)-(x_j)_3\partial_{x_j}(r_{12}^-)\bigr),\\
\textup{MYB}(3)&:=-[r_{12}^+,r_{13}^-]+[r_{12}^-,r_{23}^-]+[r_{13}^-,r_{23}^-]
-\sum_{j=1}^r\bigl((x_j)_1\partial_{x_j}(r_{23}^-)-(x_j)_2\partial_{x_j}(r_{13}^-)\bigr).
\end{split}
\end{equation}
In the following we will write 
\begin{equation*}
\begin{split}
\mathcal{A}_{ij}&:=(\sigma_\ell\otimes\tau_i\otimes \tau_j\otimes\sigma_r^*)(\mathcal{A}),
\qquad\,\,\,\, 1\leq i<j\leq n,\\
\textup{MYB}(s)_{ijk}&:=(\tau_i\otimes\tau_j\otimes\tau_k)(\textup{MYB}(s)),\qquad
1\leq i<j<k\leq n,
\end{split}
\end{equation*}
which we view as element in $\mathcal{R}\otimes\textup{End}_M((V_\ell\otimes\underline{U}\otimes V_r^*)^{\mathfrak{m}})$ (the number $n$ of $G$-representations in the tensor product representation $\underline{U}$ will always be clear from the context).
\begin{thm}\label{inteq}
For $n\geq 2$ and $1\leq i<j\leq n$, 
\begin{equation}\label{cdYBReN}
\bigl(\mathcal{A}_{ij}+\sum_{k=1}^{i-1}\textup{MYB}(1)_{kij}+
\sum_{k=i+1}^{j-1}\textup{MYB}(2)_{ikj}+\sum_{k=j+1}^n\textup{MYB}(3)_{ijk}\bigr)
\rvert_{(V_\ell\otimes\underline{U}\otimes V_r^*)^M}=0.
\end{equation}
The same holds true when 
$\widehat{\kappa}$ in $\mathcal{A}$ is replaced by $\kappa$.
\end{thm}
\begin{proof}
By a lengthy computation one shows for $1\leq i<j\leq n$,
\begin{equation}\label{commutator}
\begin{split}
[\widetilde{\mathcal{D}}_i^{\sigma_\ell,\underline{\tau},\sigma_r},
\widetilde{\mathcal{D}}_j^{\sigma_\ell,\underline{\tau},\sigma_r}]&=
\sum_{s=1}^r\bigl([\tau_j(x_s),r_{ij}^-]-[\tau_i(x_s),r_{ij}^+]\bigr)\partial_{x_s}\\
&+\mathcal{A}_{ij}+\sum_{k=1}^{i-1}\textup{MYB}(1)_{kij}+
\sum_{k=i+1}^{j-1}\textup{MYB}(2)_{ikj}+\sum_{k=j+1}^n\textup{MYB}(3)_{ijk}
\end{split}
\end{equation}
in $\mathbb{D}_M((V_\ell\otimes\underline{U}\otimes V_r^*)^{\mathfrak{m}})$. Using the explicit expression for $r^{\pm}$ one verifies that 
\[
[1\otimes h,r^{-}]=[h\otimes 1,r^{+}]\qquad
\forall\, h\in\mathfrak{a},
\] 
hence the first-order term in the right-hand side of \eqref{commutator} vanishes and we are left with
\begin{equation}\label{commutator2}
[\widetilde{\mathcal{D}}_i^{\sigma_\ell,\underline{\tau},\sigma_r},
\widetilde{\mathcal{D}}_j^{\sigma_\ell,\underline{\tau},\sigma_r}]=
\mathcal{A}_{ij}+\sum_{k=1}^{i-1}\textup{MYB}(1)_{kij}+
\sum_{k=i+1}^{j-1}\textup{MYB}(2)_{ikj}+\sum_{k=j+1}^n\textup{MYB}(3)_{ijk}
\end{equation}
in $\mathbb{D}_M((V_\ell\otimes\underline{U}\otimes V_r^*)^{\mathfrak{m}})$. The result now follows since the
left-hand side vanishes when acting on 
$C^\infty(A_{\textup{reg}};(V_\ell\otimes\underline{U}\otimes V_r^*)^M)$
by Corollary \ref{mainbKZBcor}.

Finally, replacing the role of $\widetilde{\mathcal{D}}_i^{\sigma_\ell,\underline{\tau},\sigma_r}$ by $\mathcal{D}_i^{\sigma_\ell,\underline{\tau},\sigma_r}$ gives the result when $\widehat{\kappa}$
in $\mathcal{A}$ is replaced by $\kappa$.
\end{proof}
For $n=2$ we obtain the following special case of Theorem \ref{inteq}.
\begin{cor}
The classical dynamical reflection term $\mathcal{A}$ \textup{(}see \eqref{A}\textup{)}
acts as zero on $(V_\ell\otimes U_1\otimes U_2\otimes V_r^*)^M$,
\begin{equation}\label{cdRe}
\mathcal{A}|_{(V_\ell\otimes U_1\otimes U_2\otimes V_r^*)^M}=0.
\end{equation}
\end{cor}
The equation \eqref{cdRe} is called the classical dynamical reflection equation for 
$\widehat{\kappa}$ relative to the pair $(r^+,r^-)$. 

For $n=3$ we have
\begin{cor}
The classical dynamical reflection term $\mathcal{A}$ \textup{(}see \eqref{A}\textup{)} and the mixed classical dynamical Yang-Baxter terms $\textup{MYB}(s)\in\mathcal{R}\otimes (U(\mathfrak{g})^{\otimes 3})^M$ \textup{(}see \eqref{mCYB}\textup{)}
satisfy for $1\leq i<j\leq 3$, 
\begin{equation}\label{cdYBRe}
\bigl(\mathcal{A}_{ij}+\textup{MYB}(k)\bigr)\vert_{(V_\ell\otimes U_1\otimes U_2\otimes U_3\otimes
V_r^*)^M}=0
\end{equation}
where $\{k\}=\{1,2,3\}\setminus\{i,j\}$ and $\textup{MYB}(k)$ acts in the natural way on the tensor components $U_1\otimes U_2\otimes U_3$ within $(V_\ell\otimes U_1\otimes U_2\otimes U_3\otimes V_r^*)^M$. 
\end{cor}
The three equations \eqref{cdYBRe} are called the coupled classical dynamical Yang-Baxter-reflection equations for the triple $(r^+,r^-,\widehat{\kappa})$. 

The classical dynamical reflection type equation \eqref{cdRe} does not guarantee that $\mathcal{A}_{ij}$ in \eqref{cdYBRe} vanishes when acting on $(V_\ell\otimes U_1\otimes U_2\otimes U_3\otimes V_r^*)^M$. Hence in general the reflection term in the coupled classical dynamical Yang-Baxter-reflection equations \eqref{cdYBRe} cannot be decoupled from the Yang-Baxter type terms.

\begin{defi}
We call the equations \eqref{cdYBReN} 
for $n>3$ the {\it higher-order coupled classical dynamical Yang-Baxter-reflection equations}
for the triple $(r^+,r^-,\widehat{\kappa})$.
\end{defi}
\begin{rema}\label{SSScase}
We expect that more generally,
\begin{equation}\label{cdYBReNinf}
\mathcal{A}_{ij}+\sum_{k=1}^{i-1}\textup{MYB}(1)_{kij}+
\sum_{k=i+1}^{j-1}\textup{MYB}(2)_{ikj}+\sum_{k=j+1}^n\textup{MYB}(3)_{ijk}
\end{equation}
vanishes when acting on $(V_\ell\otimes\underline{U}\otimes V_r^*)^{\mathfrak{m}}$ 
for all $n\geq 2$.
This is the case when $G$ is real split, see \cite[Thm. 6.26]{SR} and Remark \ref{SplitRemark2}.
In the real split case we have $\mathfrak{m}=0$ and hence we conclude that in this case \eqref{cdYBReNinf} vanishes identically on
$V_\ell\otimes\underline{U}\otimes V_r^*$ for all $n\geq 2$. These equations decouple, and are equivalent to the equations
\begin{equation}
\begin{split}
\mathcal{A}_{12}&=0\qquad \hbox{ in }\,\,
\mathcal{R}\otimes\textup{End}(V_\ell\otimes U_1\otimes U_2\otimes V_r^*),\\
\textup{MYB}(k)&=0\qquad \hbox{ in }\,\, \mathcal{R}\otimes\textup{End}(V_\ell\otimes
U_1\otimes U_2\otimes U_3\otimes V_r^*)\,\,\, \hbox{ for }\,\,\, k=1,2,3,
\end{split}
\end{equation}
which are the classical dynamical reflection equation for $\widehat{\kappa}$ relative
to $(r^+,r^-)$ and the mixed classical dynamical Yang-Baxter equation for $(r^+,r^-)$ from
\cite[Thm. 6.31]{SR}.
\end{rema}

\section{Example: $\textup{SU}(p,r)$.}\label{SUrp}
As an example, we work out some of the results in more detail for the real simple Lie group $\textup{SU}(p,r)$.

Write $e_{ij}^{(m)}\in\textup{GL}(m;\mathbb{C})$ for the matrix unit with a one at entry $(i,j)$ and zeros everywhere else. Write $I_m\in\textup{GL}(m;\mathbb{C})$ for the unit matrix, and $I_m^-:=\sum_{i=1}^me_{i,m+1-i}^{(m)}\in\textup{GL}(m;\mathbb{C})$ for the antidiagonal $m\times m$-matrix with ones on the antidiagonal. Fix $1\leq r\leq p$ and write
\[
C:=\left(\begin{matrix} \frac{1}{\sqrt{2}}I_r& 0 & -\frac{1}{\sqrt{2}}I_r^-\\
0 & I_{p-r} & 0\\
 \frac{1}{\sqrt{2}}I_r^- & 0 & \frac{1}{\sqrt{2}}I_r
 \end{matrix}\right)\in
 \textup{GL}(p+r;\mathbb{C})
\]
with $0$ the zero matrix of the appropriate size, and consider $G:=\{C^{-1}gC\,\, | \,\, g\in\textup{SU}(p,r)\}$. Then
\[
\mathfrak{g}_{\mathbb{R}}=\{X\in\mathfrak{sl}(p+r;\mathbb{C}) \,\, | \,\, 
X^\ast J_{p,r}+J_{p,r}X=0\}
\]
with $X^*=\overline{X}^T$ the adjoint of $X$ and
\[
J_{p,r}:=\left(\begin{matrix} 0 & 0 & -I_r^-\\
0 & I_{p-r} & 0\\
-I_r^- & 0 & 0\end{matrix}\right).
\]
The formula $\theta_{\mathbb{R}}(X):=J_{p,r}XJ_{p,r}$ ($X\in\mathfrak{g}_{\mathbb{R}}$) defines a 
Cartan involution of $\mathfrak{g}_{\mathbb{R}}$, and the diagonal matrices in $\mathfrak{g}_{\mathbb{R}}$ form 
a $\theta_{\mathbb{R}}$-stable maximally nonabelian Cartan subalgebra $\mathfrak{h}_{\mathbb{R}}$ of $\mathfrak{g}_{\mathbb{R}}$ with associated 
$\mathfrak{p}_{\mathbb{R}}$-component
\begin{equation*}
\mathfrak{a}_{\mathbb{R}}=\bigoplus_{j=1}^r\mathbb{R}(e_{jj}^{(p+r)}-e_{j^\prime j^\prime}^{(p+r)})
\quad \hbox{ with }\quad j^\prime:=p+r+1-j.
\end{equation*}
The root system $R\subset\mathfrak{h}^*$ of $\mathfrak{g}=\mathfrak{sl}(p+r;\mathbb{C})$ 
is $R=\{\epsilon_i-\epsilon_j\,\, | \,\, 1\leq i\not=j\leq p+r\}$ with $\epsilon_i(h):=\lambda_i$
for $h=\sum_{j=1}^{p+r}\lambda_je_{jj}^{(p+r)}\in\mathfrak{h}$. Take 
$R^+=\{\epsilon_i-\epsilon_j\,\, | \,\, 1\leq i<j\leq p+r\}$ as the set of positive roots. We will use the shorthand notation $\alpha_{ij}:=\epsilon_i-\epsilon_j$ for the roots. Write $\alpha_i:=\alpha_{i,i+1}$ ($1\leq i\leq p+q-1$) for the associated basis elements of $R$,
then the involution ${}^t\theta\in\textup{Aut}(R)$ is determined by
\[
\theta(\alpha_i)=
\begin{cases}
-\alpha_{i^\prime}\qquad &(1\leq i\leq r),\\
\alpha_i\qquad &(r+1\leq i<p),\\
-\alpha_{i^\prime}\qquad &(p\leq i<p+r-1),
\end{cases}
\]
which in turn determines the Satake diagram of $G$.

Write $\mathfrak{a}_{\mathbb{R}}^*=\bigoplus_{j=1}^r\mathbb{R}f_j$ with 
$f_j\in\mathfrak{a}_{\mathbb{R}}^*$ defined by $f_j(e_{ii}^{(p+r)}-e_{i^\prime i^\prime}^{(p+r)}):=\delta_{i,j}$. The restricted root system $\Sigma$ is a root system of type $\textup{BC}_r$ with the positive
restricted roots given by
\[
\Sigma^+=\{f_i\pm f_j\}_{1\leq i<j\leq r}\cup\{f_j,2f_j\}_{1\leq j\leq r}.
\]
In fact, for $1\leq i<j\leq r$ and $1\leq k\leq r$ we have
\begin{equation*}
\begin{split}
R_0&=\{\alpha_{\ell m}\}_{r+1\leq \ell\not=m\leq p},\\
R_{f_i-f_j}&=\{\alpha_{ij},\alpha_{j^\prime i^\prime}\},\qquad\quad
R_{f_i+f_j}=\{\alpha_{ij^\prime},\alpha_{ji^\prime}\},\\
R_{f_k}&=\{\alpha_{k\ell}, \alpha_{\ell k^\prime}\}_{r+1\leq \ell\leq p},\qquad
R_{2f_k}=\{\alpha_{kk^\prime}\},
\end{split}
\end{equation*}
(for $r=1$ there are no restricted roots of the form $f_i\pm f_j$ ($i\not=j$), and hence the middle line drops out). Write $\xi_i:=\xi_{f_i}$ for the multiplicative character of $A$ associated to $f_i\in\mathfrak{a}_{\mathbb{R}}^*$. The dependence on $A$ of the local factors $r^{\pm},\kappa$ of the asymptotic boundary KZB operators $\widetilde{\mathcal{D}}_{i;M}^{\sigma_\ell,\underline{\tau},\sigma_r}$  ($1\leq i\leq n$), as well as of the potential $V^{\sigma_\ell,\underline{\tau},\sigma_r}$ of the Schr{\"o}dinger operator $\widetilde{H}_M^{\sigma_\ell,\underline{\tau},\sigma_r}$, 
are then described in terms of the multiplicative characters $\xi_i\xi_j^{\pm 1}$ ($1\leq i<j\leq r$)
and $\xi_k$ ($1\leq k\leq r$), which are the characters associated to the indivisible positive restricted roots of $\Sigma$. The equivariance of the operators is with respect to the hyperoctahedral group $W\simeq S_r\ltimes (\pm 1)^r$. 

Recall that $\mathfrak{k}=\mathfrak{m}\oplus\mathfrak{q}$
with $\mathfrak{q}=\bigoplus_{\alpha\in R^+\setminus R_0^+}\mathbb{C}y_\alpha$. To further describe the local factors of $\widetilde{\mathcal{D}}_i^{\sigma_\ell,\underline{\tau},\sigma_r}$ and the
potential $V^{\sigma_\ell,\underline{\tau},\sigma_r}$ of $\widetilde{H}^{\sigma_\ell,\underline{\tau},\sigma_r}$ one needs the explicit description of $\mathfrak{m}$ and of the standard basis
elements $y_\alpha$ ($\alpha\in R^+\setminus R_0^+$) of $\mathfrak{q}$.
We have
\[
\mathfrak{m}=\mathfrak{t}\oplus\bigoplus_{r+1\leq i\not=j\leq p}\mathbb{C}e_{ij}^{(p+r)}
\]
with $\mathfrak{t}$
the $(+1)$-eigenspace of $\theta|_{\mathfrak{h}}$, which is explicitly described by
\[
\mathfrak{t}=\Bigl\{\sum_{i=1}^r\lambda_i(e_{ii}^{(p+r)}+e_{i^\prime i^\prime}^{(p+r)})
+\sum_{j=r+1}^p\mu_je_{jj}^{(p+r)} \,\, | \,\, 2\sum_{i=1}^r\lambda_i+\sum_{j=r+1}^p\mu_j=0\Bigr\}.
\]
The elements $y_\alpha$ ($\alpha\in R^+\setminus R_0^+$) are explicitly given by
\begin{equation*}
\begin{split}
y_{\alpha_{ij}}&=e_{ij}^{(p+r)}+e_{i^\prime j^\prime}^{(p+r)},\qquad
y_{\alpha_{j^\prime i^\prime}}=e_{ji}^{(p+r)}+e_{j^\prime i^\prime}^{(p+r)},\\
y_{\alpha_{ij^\prime}}&=e_{ij^\prime}^{(p+r)}+e_{i^\prime j}^{(p+r)},\qquad
y_{\alpha_{ji^\prime}}=e_{ji^\prime}^{(p+r)}+e_{j^\prime i}^{(p+r)},\\
y_{\alpha_{k\ell}}&=e_{k\ell}^{(p+r)}-e_{k^\prime\ell}^{(p+r)},\qquad
y_{\alpha_{\ell k^\prime}}=e_{\ell k^\prime}^{(p+r)}-e_{\ell k}^{(p+r)},\\
y_{\alpha_{kk^\prime}}&=e_{kk^\prime}^{(p+r)}+e_{k^\prime k}^{(p+r)}
\end{split}
\end{equation*}
for $1\leq i<j\leq r$, $1\leq k\leq r$ and $r+1\leq \ell\leq p$
(the $1^{\textup{st}}$-$4^{\textup{th}}$ line describes $y_\alpha$ for the roots $\alpha\in R_{f_i-f_j}^+$, $\alpha\in R_{f_i+f_j}^+$, $\alpha\in R_{f_k}^+$ and $\alpha\in R_{2f_k}^+$ respectively).
 
Finally, since $\mathfrak{s}(\mathfrak{gl}(p)\times\mathfrak{gl}(r))=\textup{Ad}(C)\mathfrak{k}$, the basis elements of $\mathfrak{k}$ can be expressed in terms of the standard basis elements
of $\mathfrak{s}(\mathfrak{gl}(p)\times\mathfrak{gl}(r))$, and vice versa. For instance, the center
of $\mathfrak{k}$ is spanned by 
\[
\textup{Ad}(C^{-1})\left(\begin{matrix} \frac{1}{p}I_p & 0\\ 0 & -\frac{1}{r}I_r\end{matrix}\right)=
\left(\begin{matrix} \bigl(\frac{1}{2p}-\frac{1}{2r}\bigr)I_r
 & 0 & -\bigl(\frac{1}{2p}+\frac{1}{2r}\bigr)I_r^-\\
0 & \frac{1}{p}I_{p-r} & 0\\
-\bigl(\frac{1}{2p}+\frac{1}{2r}\bigr)I_r^- & 0 & \bigl(\frac{1}{2p}-\frac{1}{2r}\bigr)I_r\end{matrix}
\right),
\]
which can be expressed as
\[
\left(\bigl(\frac{1}{2p}-\frac{1}{2r}\bigr)\sum_{i=1}^r(e_{ii}^{(p+r)}+e_{i^\prime i^\prime}^{(p+r)})+
\frac{1}{p}\sum_{j=r+1}^pe_{jj}^{(p+r)}\right)-\Bigl(\frac{1}{2p}+\frac{1}{2r}\Bigr)\sum_{i=1}^ry_{\alpha_{ii^\prime}},
\]
where the first term lies in $\mathfrak{t}$ and the second term lies in $\mathfrak{q}$.

\bibliographystyle{amsplain}

\end{document}